\documentclass[11pt]{amsart}
\usepackage{amsmath, amsfonts, amssymb,amsthm}
\usepackage{euscript}
\usepackage[T1]{fontenc}
\usepackage{graphicx}
\usepackage{color}

\newtheorem{thm}{Theorem}[section]
\newtheorem{lem}[thm]{Lemma}
\newtheorem{prop}[thm]{Proposition}
\newtheorem{cor}[thm]{Corollary}

\theoremstyle{definition}
\newtheorem{defn}[thm]{Definition}

\newtheorem{ex}[thm]{Example}
\newtheorem{rem}[thm]{Remark}
\newtheorem{Q}[thm]{Question}

\DeclareMathOperator{\R}{\mathbb R}

\DeclareMathOperator{\N}{\mathbb N}
\DeclareMathOperator{\Z}{\mathcal Z}
\DeclareMathOperator{\I}{\mathcal I}
\DeclareMathOperator{\D}{\mathcal D}
\DeclareMathOperator{\SR}{\mathcal R}
\DeclareMathOperator{\SO}{\mathcal O}
\DeclareMathOperator{\QQ}{\mathbb Q}

\DeclareMathOperator{\dom}{dom}
\DeclareMathOperator{\ord}{ord}
\DeclareMathOperator{\Rad}{Rad}
\DeclareMathOperator{\Sp}{Spec}
\DeclareMathOperator{\pol}{indet}

\def\RadRe {\rm{RadRe}}
\def\Rad {\rm{Rad}}
\def\RadRe {\sqrt[R]}
\def\Rad {\sqrt}

\def \Zd {{\mathcal Z}}

\def \RR {{\mathbb R}}
\def\air{\vskip 1cm}

\begin{document}

\title[\tiny{Continuous functions in the plane regular after one blowing-up}]{Continuous functions in the plane regular after one blowing-up}

\author[G.~Fichou, J.-P.~Monnier, R. Quarez]{Goulwen Fichou,
  Jean-Philippe Monnier, Ronan Quarez}

\thanks{The third author is supported by French National Research Agency (ANR) project GEOLMI - Geometry and Algebra of Linear Matrix Inequalities with Systems Control Applications}

\address{Goulwen Fichou\\
IRMAR (UMR 6625), Universit\'e de
         Rennes 1\\
Campus de Beaulieu, 35042 Rennes Cedex, France}
\email{goulwen.fichou@univ-rennes1.fr}

\address{Jean-Philippe Monnier\\
   LUNAM Universit\'e, LAREMA, Universit\'e d'Angers}
\email{jean-philippe.monnier@univ-angers.fr}

\address{Ronan Quarez\\
IRMAR (UMR 6625), Universit\'e de
         Rennes 1\\
Campus de Beaulieu, 35042 Rennes Cedex, France}
\email{ronan.quarez@univ-rennes1.fr}

\date{\today}

\maketitle

\begin{quote}\small
\end{quote}

\begin{abstract}
We study rational functions admitting a continuous extension to the
real affine space. First of all, we focus on the regularity of such
functions exhibiting some nice properties of their partial
derivatives. Afterwards, since these functions correspond to rational
functions which become regular after some blowings-up, we work on the
plane where it suffices to blow-up points and then we can count the
number of stages of blowings-up necessary. In the latest parts of the
paper, we investigate the ring of rational continuous functions on the
plane regular after one stage of blowings-up. In particular, we prove
a Positivstellensatz without denominator in this ring.
\end{abstract}

\section{Introduction}
In real algebraic geometry, the choice of a good class of functions is an important matter. From polynomial functions to semialgebraic ones, passing through rational functions, regular functions (rational without poles) or Nash ones (real analytic and semialgebraic), a large class of different functions with a specific flavor is available. In this paper we focus on the class of continuous functions lying between rational and regular functions, with $\frac{x^3}{x^2+y^2}$ as a typical example among the classical functions appearing in calculus courses.
The rational functions admitting a continuous extension at their poles
have known a recent interest in the work of Kucharz on the good way to
approximate continuous maps between spheres \cite{Ku}. The surprising
behaviour of rational continuous functions on singular sets has been
exhibited by Koll\'ar (\cite{Ko} or \cite{KN}), whereas their
systematic study on smooth varieties have been performed in
\cite{FHMM}. More recently Kucharz and Kurdyka discussed real
algebraic vector bundle using this approach  \cite{KuKu}. Notably, we
know from \cite{Ko} that the restriction of a rational continuous
function remains rational, provided that the ambient space is
smooth. We know also that such functions on smooth varieties are
exactly those continuous functions that become regular after
blowings-up \cite{FHMM}. Note that rational continuous functions appear also naturally in the algebraic identities which answer the 17th Hilbert Problem, namely Kreisel noticed that any non-negative polynomial is a sum of squares of rational continuous functions \cite{Kr}.

\vskip 5mm

In the present paper, we tackle two natural questions related to these rational and continuous functions on affine spaces, which we call regulous following the terminology introduced in \cite{FHMM}. The first one concerns the regularity of regulous functions. Note first that regulous functions admit first order partial derivatives in any direction as a consequence of the restriction property (Proposition \ref{derpartregulu}).  This automatic existence enables to give a nice description of a regulous function of class $C^k$, or $k$-regulous function, just in terms of the continuity of the $k$-partial derivatives of the associated rational function (Theorem \ref{thm-top}). 
Note however that the mixed partial derivatives of order two do not exist in general, or that the first partial derivatives are not necessarily locally bounded around the poles.

Another approach to discuss the behaviour of regulous functions is to focus on the number of stages of blowings-up necessary to make it regular. We restrict our attention to the planar setting where it suffices to perform blowings-up along points. We count the number of stages by taking into account the maximal chain of infinitely near points. For instance, the function $\frac{x^3}{x^2+y^2}$ becomes regular after one stage of blowing-up whereas the function
$\frac{x^3}{x^2+y^4}$ needs two stages of blowings-up to become regular.
The regulous functions which become regular after one stage of blowings-up (or equivalently after the blowing-up of a finite number of points in the plane) are particularly nice: these functions, called $[1]$-regulous, admit a simple local characterisation in terms of positive definiteness of their denominator (Theorem \ref{equivfctregapres1eclat}). This characterisation enables to prove the local boundedness of the first order particular derivatives (Proposition \ref{locbound}) and to guaranty their $C^k$ regularity just by looking at their polynomial expansion of order $k$ (Theorem \ref{question3})!

\vskip 5mm

We know already from \cite{FHMM} that there is no difference between the topologies induced by $k$-regulous functions for different values of $k$ in $\mathbb N$. If it is still the case when a number of stages of blowings-up is fixed, we prove however that the situation changes when, insides $k$-regulous functions, we vary the number of stages (cf. section \ref{topo}). 

\vskip 5mm

The good properties of the rings $\SR_{[l]}^{k}(\R^2)$ of $k$-regulous functions regular after $l$ stages of blowings-up in the plane lead us to study their real algebraic properties in comparison with those of the rings $\SR^{k}(\R^2)$ of $k$-regulous functions. The so-called \L ojasiewicz property in \cite[Lem. 5.1]{FHMM} was the key tool to prove the weak, real and ``strong'' Nullstellens\"atze in the rings $\SR^{k}(\R^n)$. Unfortunately, this property is no more true in the rings $\SR^{k}_{[1]}(\R^2)$.
Thus, we are lead to give a new proof of the Nullstellensatz that does not use the \L ojasiewicz property and which extends to
the rings $\SR^{k}_{[l]}(\R^2)$ to get at least a {\it real} Nullstellensatz (since the Nullstellensatz appears not to be valid). Moreover, we prove the radical principality of the rings $\SR^{k}_{[1]}(\R^2)$ which can be seen as a weak \L ojasiewicz property.

We complete the study of real algebraic properties by mentioning that any boolean combination of sets given by the positivity locus $\{f>0\}$ of a given function $f$ in $\SR^{k}_{[l]}(\R^2)$ is a semialgebraic subset of $\R^2$. Combined with Tarski-Seidenberger Theorem, this helps us to study the real spectrum of the rings $\SR^{k}_{[l]}(\R^2)$ and leads us to an useful Artin-Lang property.

\vskip 5mm

Finally, the techniques developed all along the paper enable to bring
a new enlightenment on Hilbert 17th problem. We recall that the
Hilbert 17-th problem (answered by the affirmative by Artin in 1927)
asked whether a non-negative polynomial is a sum of squares of rational functions. This kind of algebraic certificates of non-negativity is also called a Positivstellensatz.
Versions of Hilbert 17-th problem have been studied in several
geometric other settings: in the ring of analytic functions (\cite{Ru}
and  \cite{Ri}), of Nash functions (\cite[8.5.6]{BCR}), etc. 
In general, the considered non-negative function is a sum of squares only
in the associated fraction field: for example the Motzkin polynomial
is non-negative 
but not a sum of squares of polynomials. If it is sometimes possible
to get Positivstellens\"atze {\it without denominator} 
(namely when the sum of squares lies already in the ring), these cases are quite rare.\par
This motivates us to look at the 17-th Hilbert problem in the rings of
regulous functions. We produce a Positivstellensatz without
denominator in $\SR^0(\R^n)$ (Proposition \ref{sosreg0}) 
using a classical argument, namely the formal Positivstellensatz
\cite[Prop. 4.4.1]{BCR} and the Artin-Lang property of section
\ref{RadicPrincip}. 
We also get a Positivstellensatz without denominator in the ring of
regular functions on $\R^2$ (Proposition \ref{reguluesosdim2}), 
using the properties of so-called ``bad points'' introduced by Delzell \cite{De}.\par 
The heart of section \ref{Hilbert17} is to show that there exists also
a Positivstellensatz without denominator in $\SR^0_{[1]}(\R^2)$
(Theorem \ref{SOS})). 
Namely, we show that any such rational function which is non-negative everywhere on $\R^2$, can be written as a sum of squares in $\SR^0_{[1]}(\R^2)$.
Note that this time a formal argument does not work any longer. We first consider the case where our non-negative function $f$ vanishes at its set of poles (a condition which can be seen as a flatness condition); in that case one gets easily the result. Then, in the general case, we give a constructive proof (assuming we know how to write a polynomial as a sum of squares of rational functions). The number of steps of the proof is given by the cardinality of the set of poles where $f$ does not vanish, a set which somehow measure the defect of flatness for $f$. 

Note that our method does not extend to the rings $\SR^k_{[l]}(\R^2)$ for any integers $(k,l)\not=(0,1)$. 

\section{Partial derivatives of regulous functions} 

Let $n$ be an integer and let $X\subset \R^n$ be an irreducible,
smooth, affine variety. Let $f\in \R(X)^*$ be a rational function on
$X$. The domain of $f$, denoted by $\dom (f)$, is the biggest Zariski open subset of $X$ on
which $f$ is regular, namely $f=\dfrac{p}{q}$ where $p$ and $q$ are polynomial
functions on $\R^n$ such that $q$ does not vanish on $\dom (f)$. The indeterminacy locus of $f$ is defined to be the Zariski closed set $\pol(f)=X\setminus \dom(f)$. 

\subsection{Regulous functions}

We recall the definition of a $k$-regulous function given in \cite{FHMM}.
\begin{defn}
\label{def1}
Let $n$ be a positive integer and let $X\subset \R^n$ be an irreducible,
smooth, affine variety.
Let $k\in\N\cup\{\infty\}$. W say that a function $f:X\to \R$ is $k$-regulous on $X$ if  $f$ is $C^k$ on $X$ (with its induced structure of a
$C^{\infty}$-variety) and $f$ is a rational function on $X$, i.e. there exists a
non-empty Zariski open subset $U\subseteq X$ such that $f|_U$ is
regular.\par
A $0$-regulous function on $X$ is simply called a regulous function.
\end{defn}

Let $k\in \N\cup\{\infty\}$.
We denote by
$\SR^k(X)$ the ring of $k$-regulous functions on $X$. By Theorem 3.3
of \cite{FHMM} we know that $\SR^{\infty}(X)$ coincides with the ring
$\SO (X)$ of regular functions on $X$.

Denote by $\Z(f)$ the zero set of the real function $f$.

\begin{lem}
\label{lemme1}
Let $f\in \SR^0 (\R^n)$ and suppose $f=\frac{p}{q}$ on $\dom(f)$ with
$p$ and $q$ some real polynomials in $n$ variables. Then
$\Z(q)\subset \Z(p)$ and $\Z(f)\subset \Z(p).$
\end{lem}

\begin{proof}
It follows from the identity $qf=p$ in $\SR^0(\R^n)$.
\end{proof}

Note that a rational function $f\in  \R(X)^*$ on $X$ is
regulous if and only if $f$ can be extended continously to $\pol(f)$,
because $X$ being smooth, the euclidean closure of a non-empty Zariski
open subset $U\subseteq X$ is equal to $X$.\par

We will propose latter a similar characterisation of a $k$-regulous function. 

\begin{defn}
\label{definitionordre}
Let $x\in X$. Let $\mathbf m_x$ be the maximal ideal of $\SO (X)$
corresponding to $x$.
\begin{enumerate}
\item Let $p\in\SO(X)$. We denote by $\ord_x(p)$ the order of $p$ at
  $x$, this is the biggest integer $l$ such that $p\in {\mathbf m}_x^l$.
\item Let $f=\dfrac{p}{q}\in \R(X)^*$ a rational function on $X$ such
  that $p\in\SO(X)$ and $q\in\SO(X)$. We define the order of $f$ at
  $x$, denoted by $\ord_x(f)$, as the integer $\ord_x(p)-\ord_x(q)$ 
\end{enumerate}
\end{defn}

We state below some results that will be useful in the sequel.

\begin{prop}
\label{dimension1}
Let $n$ be an integer and let $X\subset \R^n$ be an irreducible,
smooth, affine variety of dimension $1$. Then, for $k\in\N$, the ring $\SR^k (X)$ of $k$-regulous functions on $X$ coincides with the ring $\SR^{\infty} (X)$ of regular functions on $X$.
\end{prop}

\begin{proof} 
Note that it is sufficient to prove the inclusion
$$\SR^0 (X)\subset \SR^{\infty} (X).$$
Take $f\in\SR^0 (X)$. 
There exist $p,q\in \SO (X)$ such that $f=\dfrac{p}{q}$ on $\dom (f)$, with $q(x)\not=0$ for any $x\in \dom(f)$.

 Set $l=\ord_x(q)$ and $m=\ord_x(p)$. Choose $h\in \SO(X)$ to be a uniformizing parameter of
the discrete valuation ring $\SO_{x,X}$, so that 
$\dfrac{p}{h^m}(x)\not=0$ and $\dfrac{q}{h^l}(x)\not=0$. Since
$\dom(f)\cap \D(h)$
is dense in $X$ for the euclidean topology, the point $x$ is the limit of a sequence of  points
$x_i\in \dom(f)\cap \D(h)$. We get 
$$\dfrac{q}{h^l}(x)f(x)=\lim \dfrac{q}{h^l}(x_i)f(x_i)=\lim
\dfrac{p}{h^l}(x_i)$$ $$=\lim
\dfrac{p}{h^m}(x_i)h^{m-l}(x_i)=\dfrac{p}{h^m}(x)\lim h^{m-l}(x_i).$$ 
We have proved that $m\geq l$, meaning that $f$ is regular at $x$.
\end{proof}

In analogy with the arc-analytic functions defined by K. Kurdyka \cite{KK}, we say that a function $f:\R^n\to \R$ is \emph{arc-algebraic} if the restriction of $f$ to any smooth, connected, euclidean open
subset $U$ of an irreducible curve $X\subset\R^n$ coincides on $U$
with the restriction of a rational function on $X$ regular on $U$.

\begin{cor}
\label{restrictionouvertcourbe}
Let $f\in\SR^0 (\R^n)$. Then $f$ is arc-algebraic.
\end{cor}

\begin{proof} Let $X\subset\R^n$ be an irreducible curve and $U\subset X$ be a smooth, connected, euclidean open subset of $X$.
By \cite{Ko} or \cite{FHMM}, the restriction of $f$ to $X$ coincide
with a rational function $g$ on $X$ which can be extended continuously
to the whole of $X$. Even if $X$ may be singular, we can follow the proof of Proposition \ref{dimension1} to show that $g$ is regular on the smooth open subset $U$ of $X$.
\end{proof}

\begin{rem} The smoothness of the set $U$ is crucial. For example, the plane curve ${\mathcal C}$ given by the equation $y^3+2x^2y-x^4=0$ is analytically smooth but admits the origin as a singular point (cf. \cite{BCR} p.69). In that case, the rational function $\frac{x^2}{y}$ is of class $C^{\infty}$ on ${\mathcal C}$, since it satisfies $\frac{x^2}{y}=1+\sqrt{1+y}$ on a neighbourhood of the origin. In particular $\frac{x^2}{y}$ is not arc-algebraic at the origin.
\end{rem}

\subsection{Partial derivatives}
Let $f$ be a rational function on $\R^n$, $a\in\R^n$ and $v$ be a vector of $\R^n$. 
We denote by $\partial_v f$ the partial derivative of
$f$ in the direction $v$. Namely 
$$\partial_v f(a)=\lim_{t\to 0}\frac{f(a+tv)-f(a)}{t}.$$
In particular $\partial_{x_i} f$ denotes the partial derivative of $f$ in the direction of the $x_i$-axis.
Likewise, for any positive integer $k$, we denote the $k$-th partial derivative of $f$ with respect successively to 
$x_{i_k}\cdots x_{i_1}$ by 
$\partial^k_{x_{i_1}\cdots x_{i_k}}f$. 
\par
Then, any $k$-th partial derivative  $\partial^k_{x_{i_1}\cdots x_{i_k}}f$ is a rational function on $\R^n$ which is
$C^{\infty}$ on $\dom (f)$. Note moreover that
 $$\dom (f)\subseteq \dom \left(\partial^k_{x_{i_1}\cdots x_{i_k}}f\right).$$

Let $f\in\SR^0 (\R^n)$. We can see $f$ as a continuous function on
$\R^n$ which is rational but also as a regular function on $\dom (f)$
which can be extended continuously to the whole of $\R^n$. The partial
derivative functions $\partial_{x_i}f$ of 
$f$ are regular functions on $\dom (f)$ but, a priori, they are not
defined on $\pol (f)$. In the following proposition, we prove that the
partial derivatives  of a regulous function are automatically defined on
the whole of $\R^n$.

\begin{prop}
\label{derpartregulu}
Let $f\in \SR^0 (\R^n)$. For any $a\in\R^n$, the function $f$ admits a partial
derivative in any direction at $a$.
\end{prop}

\begin{proof}
Assume $a\in\pol(f)$, otherwise the result is immediate. Let $v\in \R^n$. By
Corollary \ref{restrictionouvertcourbe}, the one variable function $g:t\mapsto f(a+tv)$ 
is $C^{\infty}$, so the partial derivative ${\partial}_{v}f$ of $f$ at $a$ 
exists and is equal to $g'(0)$.
\end{proof}

\begin{rem}
\begin{enumerate}
\item The previous proposition does not mean that $f$ is $C^1$ as is shown by the following example.\par 
 Let $f$ be the two-variables regulous function $f(x,y)=\dfrac{x^3}{x^2+y^2}$. Then $f(t,0)=t$ therefore ${\partial}_{ x}f(0,0)=1$. Note that ${\partial}_{x}f(x,y)=\dfrac{x^4+3x^2y^2}{(x^2+y^2)^2}$ therefore $f$ is not $C^1$.
\item 
We see that for any regulous function, all partial derivatives of the form $\partial^k_{v^k}f$ exist for any $k\in \N$. Note however that for the two-variables regulous function $f(x,y)=\dfrac{x^3}{x^2+y^2}$, the second partial derivative $ {\partial}_{\!y x}^2f(0,0)$ does not exist.
\end{enumerate}
\end{rem}

We state some corollaries of the existence of the partial derivatives of a regulous function. First, an immediate consequence for $k$-regulous functions.

\begin{cor} Let $k$ be an integer and let $f\in \SR^k (\R^n)$. The
  partial derivatives of order $k+1$ of $f$ are defined on
  $\R^n$.
\end{cor}

Second, this property enables to reinterpret the definition of a $k$-regulous function.

\begin{cor}
\label{equivalence1}
Let $k$ be an integer and let $f\in \SR^0 (\R^n)$.
Then $f\in \SR^{k+1}(\R^n)$ if and only if
the partial derivatives ${\partial}_{x_i}f$ of
$f$ are
$k$-regulous functions on $\R^n$ for any $i\in\{1,\ldots,n\}$ or equivalently if the partial derivatives ${\partial}_{x_i}f$, for $i\in\{1,\ldots,n\}$,
are $C^k$ on $\R^n$.
\end{cor}

\begin{proof} Assume first $f\in \SR^{k+1}(\R^n)$. The
partial derivative functions ${\partial}_{ x_i}f$ are thus
rational functions of class $C^k$ on $\R^n$, and hence $k$-regulous.\\
Assume the partial derivative functions ${\partial}_{ x_i}f$ of $f$ are
$k$-regulous functions on $\R^n$, the function $f$ is then of class
$C^{k+1}$ on $\R^n$. Since $f\in \SR^0 (\R^n)$ then the function $f$
is rational and hence $f\in \SR^{k+1} (\R^n)$.
\end{proof}

Finally, the automatic existence of the first partial derivatives of a regulous function enables to give a characterisation of $1$-regulous functions in terms of the behaviour of their first partial derivatives with respect to continuity. 

\begin{cor}\label{equivalence2} Let $f\in \SR^0 (\R^n)$. Then $f\in \SR^1(\R^n)$ if and only if the partial derivatives ${\partial}_{x_1}f, \ldots,{\partial}_{x_n}f$, considered as rational functions on $\dom(f)$, can be extended continuously to $\R^n$.
\end{cor}

In particular, we don't need to pay attention to the values on $\pol(f)$ of the partial derivative functions of a regulous function $f$.

\begin{proof}[Proof of Corollary \ref{equivalence2}] Let $p$ and $q$ be polynomial functions on $\R^n$ such that the rational function $\frac{p}{q}$ is well-defined and coincides with $f$ on $\dom(f)$.

If $f\in C^1(\R^n)$, then for $i\in \{1,\ldots,n\}$, the partial derivative ${\partial}_{x_i}f$ is continuous on $\R^n$. As a consequence, the rational function ${\partial}_{ x_i}\left({\dfrac{p}{q}}\right)$, well-defined and equal to ${\partial}_{ x_i}f$ on $\dom(f)$, admits a continuous extension to $\R^n$. 

Conversely, we need to prove that, for $i\in \{1,\ldots,n\}$, the
continuous extension of the partial derivative ${\partial}_{ x_i}f$ at
any $a\in \pol(f)$ coincides with the value of ${\partial}_{ x_i}f$ at
$a$. Denote by $l$ the limit of ${\partial}_{ x_i}f(b)$ as $b$ tends
to $a$ in $\dom(f)$. If the line passing through $a$ with direction
$x_i$ is not included in $\pol(f)$ in a neighbourhood of $a$, then $l$ is in particular the limit at 0 of the derivative of the one variable function $g:t\mapsto f(a+te_i)$, where $(e_1,\ldots,e_n)$ stands for the canonical basis of $\R^n$, which is $C^{\infty}$ by arc-algebraicity. So $l={\partial}_{ x_i}f(a)$, and therefore ${\partial}_{ x_i}f$ is continuous at $a$.

In case the line passing through $a$ with direction $x_i$ is included
in $\pol(f)$ in a neighbourhood of $a$, we choose another coordinates
system $(y_1,\ldots,y_n)$ on $\R^n$ at $a$ so that none of the
$y_j$-axis are locally included in $\pol(f)$ (which is a codimension at least two subset of $\R^n$, cf. \cite{FHMM}). We obtain like this the continuity of the partial derivatives ${\partial}_{y_1}f, \ldots,{\partial}_{ y_n}f$ at $a$, and therefore the continuity of ${\partial}_{x_i}f$.
\end{proof}

We generalize the previous result to $k$-regulous functions.

\begin{prop}
\label{equivalence3}
Let $k$ be a positive integer and let $f\in \SR^0 (\R^n)$.
Then $f\in \SR^{k}(\R^n)$ if and only if for any $j\in \{1,\ldots,k\}$
the $j$-th partial derivatives ${\partial^j}_{\!
  x_{i_1}\cdots x_{i_j}}f$, seen on $\dom (f)$, can be extended continuously to $\pol(f)$.
\end{prop}

\begin{proof}
The direct implication is clear. We prove the converse implication by induction on $k$. The case $k=1$ has been proved in Corollary \ref{equivalence2}. Assume now the induction property for some $k\geq 1$. Assume that for any $j\in \{1,\ldots,k+1\}$ the $j$-th partial derivatives of $f$ can be extended continuously to $\pol(f)$. 
Then the function $f$ is $k$-regulous by induction hypothesis, so that the $k$-th partial derivatives of $f$, which are rational functions on $\dom(f)$, are
regulous on $\R^n$.
Moreover, their partial derivatives can be extended
continuously to $\pol(f)$, therefore the $k$-th partial derivative functions are $1$-regulous functions on $\R^n$ by Corollary \ref{equivalence2} again. As a consequence
 $f$ is $k+1$-regulous on $\R^n$.
\end{proof}

Another formulation of Proposition  \ref{equivalence3}.
\begin{prop}
\label{equivalence3bis}
Let $k$ be an integer and let $f\in\R(x_1,\ldots,x_n)$.
Then $f\in \SR^{k}(\R^n)$ if and only if for any $j\in \{0,\ldots,k\}$
the $j$-th partial derivatives ${\partial^j}_{\!
  x_{i_1}\cdots x_{i_j}}f$, seen on $\dom (f)$ ($\partial^0 f=f$), can be extended continuously to $\pol(f)$.
\end{prop}

We give now an application of Proposition \ref{equivalence3}.

\begin{cor}
\label{arithmetique}
Let $k\in\N$ and let $f\in \SR^k (\R^n)$. Choose 
$p,q\in \R [x_1,\ldots,x_n]$ such that $f=\dfrac{p}{q}$ on $\dom(f)$. Then 
$$\forall l\in\N,\,\forall t\in\N^*,\,\,p^lf^t\in\SR^{k+l}(\R^n).$$
\end{cor}

\begin{proof}
We prove the result by induction on $l$. For $l=0$ the result is immediate. Assume the induction property for some $l\in \N$. For $t\in\N^*$, consider the
partial derivative ${\partial}_{ v} (p^{l+1}f^t)$ of $p^{l+1}f^t$ in the direction $v\in \R^n$ on $\dom (f)$. We have 
$$ {\partial}_{ v} (p^{l+1}f^t)={\partial}_{ v}\left(\dfrac{p^{l+t+1}}{q^t}\right)=
(l+t+1)\left({\partial}_{ v} p\right)\dfrac{p^{l+t}}{q^t}-p^{l+t+1}\dfrac{t\left({\partial}_{ v}
  q\right) q^{t-1}}{q^{2t}}$$
$$=(l+t+1)\left({\partial}_{ v} p\right) p^l f^t-t \left({\partial}_{ v}
  q\right) p^l f^{t+1}.$$
By the induction hypothesies, $p^l f^t$ and $p^l f^{t+1}$ are
$k+l$-regulous on $\R^n$. Consequently, for $j\in \{1,\ldots,k+l+1\}$, any 
$j$-th partial derivative function of $p^{l+1}f^t$ on $\dom (f)$
can be extended continuously to $\pol(f)$. By Proposition
\ref{equivalence3}, the function $p^{l+1}f^t$ is ${(k+l+1)}$-regulous. 
\end{proof}

Consider the continuous function
  $f:\R\rightarrow\R$, $x\mapsto |x|$. Then $f''$ can be extended
  continuously at $0$ but of course $f$ is not $C^2$ on $\R$. We prove that for regulous functions, this case can not appear, namely it is enough to consider only the partial derivative of maximal order. We note moreover that the continuity of $f$ is not even necessary. 

\begin{thm}\label{thm-top}
Let $k$ be a positive integer and let $f\in \R(x_1,\ldots,x_n)$ be a rational function with $\pol(f)$ of codimension at least two in $\R^n$.
Then $f\in \SR^{k}(\R^n)$ if and only if 
all the $k$-th partial derivatives ${\partial}_{\!
  x_{i_1}\cdots x_{i_k}}^kf$, seen on $\dom (f)$, can be extended
continuously to $\pol(f)$ for any $i_1,\ldots,i_k\in\{1,\ldots,n\}$.
\end{thm}

Before the proof we state a key lemma.

\begin{lem}\label{lem-top} Let $f\in \R(x_1,\ldots,x_n)$ be a rational function with $\pol(f)$ of codimension at least two in $\R^n$. Assume that the first partial derivatives of $f$ can be extended continuously to $\R^n$. Then $f$ can be extended continuously to $\R^n$ and the extension is $1$-regulous.
\end{lem}

\begin{proof}
It is sufficient to prove that $f$ is regulous by Corollary
\ref{equivalence2}. Let $x\in \pol(f)$, and let us prove that $f$ is
continuous at $x$. The absolute values of the continuous extensions of
the first partial derivatives are locally bounded around $x$, say by
$K>0$ on a small convex neighbourhood $U$ of $x$ in $\R^n$.
We are going to prove that there exists $K'\in \R_+^*$ such that, for $a,b\in \dom(f)\cap U$, we have
$$|f(a)-f(b)| \leq  K'\Vert a-b\Vert_1$$
where $\Vert\cdot\Vert_1$ denotes the $L^1$ norm on $\R^n$.\par 
Once this inequality is proved, we obtain that the limit of $f$ at $x$ exists. Indeed, first $f$ is locally bounded around $x$, and second if $(\alpha_i)$ and $(\beta_i)$ are sequences converging both to $x$ such that $f(\alpha_i)$ converges to $l_1$ and $f(\beta_i)$ converges to $l_2$, then the inequality implies $l_1=l_2$.

Let us prove the announced inequality. There exists a plane $P$
containing $a$ and $b$ and intersecting $\pol(f)$ in a finite number
of points since $\pol(f)$ has codimension two in $\R^n$. Then, we join
$a$ to $b$ by the segment $[a,b]$ if $[a,b]\cap \pol (f)$ is empty, or
by the two small sides of a right-angle triangle included in $P\cap U$ with hypotenuse $[a,b]$ otherwise. 
In the first case, we parametrize the segment $[a,b]$ by $\gamma: t\mapsto ta+(1-t)b$. The function $g:t\mapsto f\circ \gamma$ is $C^{1}$ by composition, and
$$g'(t)=\sum_{i=1}^n (a_i-b_i)({\partial_{x_i}}f) (\gamma(t)).$$
Then
$$|f(a)-f(b)|=|g(1)-g(0)|\leq \sup_{t\in [0,1]} |g'(t)| \leq   K \Vert a-b\Vert_1$$
by the mean value theorem. 
In the second case, denoting by $c$ the third vertex of the triangle, we obtain similarly
$$|f(a)-f(b)| \leq   K (|a-c|_1+|c-b|_1)\leq \sqrt{2n}K\Vert a-b\Vert_1.$$
\end{proof}

\begin{proof}[Proof of Theorem \ref{thm-top}] Assume that the $k$-th partial derivatives  of $f$, seen on $\dom (f)$, can be extended
continuously to $\pol(f)$. Then the $(k-1)$-th partial derivatives of $f$ satisfy the condition in Lemma \ref{lem-top}, so that they can be extended continuously on 
$\R^n$. Iterating the process $k-2$ times, we obtain that $f$ is $k$-regulous by Proposition \ref{equivalence3}.
\end{proof}

\subsection{Flatness}
The notion of $k$-flatness, for $k\in \N$, will play an important role in the paper when discussing about sum of squares, (see Proposition \ref{plateapresuneclat} and the proof of Theorem \ref{SOS}).

\begin{defn} Let $f:\R^n\to \R$ be a real function and $k\in \N$ be an
  integer. We say that $f$ is $k$-flat at a point $a\in \R^n$ if $f$
  is of class $C^k$ at $a$ and if all partial derivatives of $f$ of
  order less than or equal to $k$ are equal to zero at $a$. We say that $f$ is $k$-flat if $f$ is $k$-flat at any point of its zero set.
\end{defn}

A key tool to discuss about flatness of regulous functions is the \L
ojasiewicz property \cite[Lem. 5.1]{FHMM}. It says that, given $f\in
\mathcal R^k(\R^n)$ and $g$ $k$-regulous on $\R^n\setminus \mathcal
Z(f)$, there exists an integer $N\in \N$ such that the function
$f^Ng$, extended by zero at $\mathcal Z(f)$, is $k$-regulous on $\R^n$.

Using the \L ojasiewicz property, we show below that a sufficiently
big power of a regulous function is $k$-flat.

\begin{prop}\label{kflat2} Let $k$ and $n$ be positive integers and $f\in \mathcal R^0(\R^n)$. There exists $N\in \N$ such that for all integers $m\geq N$ the function $f^m$ is $k$-flat.
\end{prop}

\begin{proof}
Write $f=\frac{p}{q}$ with $p$ and $q$ coprime polynomials.  
By \cite[Lem. 5.1]{FHMM}, there exists $M$ such that $f^M.\frac{1}{q^{2^k}}$ can be extended continuously by $0$ on $\Z(f)$. As a consequence, the function 
$f^M.\partial^j_{v_1\cdots v_j}
f$ can be extended continuously by $0$ on $\Z(f)$, for any $j\in \{1,\ldots,k\}$ and $v_1,\ldots,v_j\in \R^n$. 

Set $N=kM+k$. Consider an integer $m\geq N$ and vectors $v_1,\ldots,v_k\in \R^n$. Using the classical derivation rules, and writing the regular function
$\partial^k_{v_1\cdots v_k} f^m$ on $\dom(f)$ as a sum of some powers of $f$ times
some $j$-th partial derivative functions of $f$ with $1\leq j\leq k$,
we see that $\partial^k_{v_1\cdots v_k} f^m$ can
be extended continuously by $0$ on $\Z(f)$. Then the function $f^m$ is
of class $C^k$  on $\Z(f)$ by Theorem \ref{thm-top}, and moreover $f^m$ is $k$-flat.
\end{proof}

We obtain even more if we assume than the poles of the regulous function are also zeroes.

\begin{cor} Let $k$ and $n$ be positive integers and $f\in \mathcal R^0(\R^n)$. Assume $\pol(f)\subset \mathcal Z(f)$. 
There exists $N\in \N$ such that for all integers $m\geq N$ the function $f^m$ is of class $C^k$ on $\R^n$ and $f^m$ is $k$-flat.
\end{cor}

\begin{proof}
The function $f$ is of class $C^{\infty}$ on $\R^n\setminus \pol(f)$, so the result is a direct consequence of Proposition \ref{kflat2}. 
\end{proof}

\section{Regulous functions on $\R^2$}

We have seen up to now that the partial derivatives of a regulous function have a very nice behaviour. We continue to investigate which good properties these functions may satisfy. Let us raise some basic questions.

\begin{Q}\label{ques}
\begin{flushleft}\end{flushleft}
\begin{enumerate}
\item Let $f\in \SR^0 (\R^n)$ be a regulous function. Are the partial derivatives of $f$ locally bounded at any point $a\in\pol(f)$?
\item Let $k\in\N$. Let $f\in\SR^k
(\R^n)$ and $h\in\SR^{k+1}(\R^n)$ be regulous functions such that the indeterminacy locus $\pol(f)$ of $f$ is included in the zero set $\Z (h)$ of $h$. Do we
get that the product $hf$ belongs to $\SR^{k+1}(\R^n)$?
\item Let $k\in\N$ and let $f\in\SR^0 (\R^n)$ be a regulous function. Do we have an equivalence between the fact that $f$ is $k$-regulous and the property that
 $f$ has a polynomial expansion of order $k$ at any point of $\R^n$, i.e. $\forall a\in\R^n$, there exists a polynomial
$P_k\in\R[x_1,\ldots,x_n]$ of degree less than or equal to $k$ such that
$$f(a+h)=f(a)+P_k(h)+o(||h||^k)$$
for $h$ close to zero?
\end{enumerate}
\end{Q}

Next examples show that all these questions have a negative answer.

\begin{ex}
\begin{flushleft}\end{flushleft}
\begin{enumerate}
\item Let $f:\R^2\to \R$ be the regulous function defined by 
$$f (x,y)=\left\{ \begin{array}{ll}
\dfrac{yx^2}{x^2+y^4}\,\,\,{\rm if}\,\, (x,y)\not=(0,0)  \\
0\,\,\,~~~~~~~~~~~~~~~~~~~~~~{\rm ~~~~~~~~~~~~~~~~~~~~~~~~otherwise}
\end{array}\right.$$
Then ${\partial}_{ x}f(x,y)=\dfrac{2xy^5}{(x^2+y^4)^2}$
for $(x,y)\not=(0,0)$, is not bounded along the arc given by $x=t^2$ and $y=t$.
\item Keep $f$ as in the first example, and choose for $h$ the function $h=y$. Denote by $g$ the product $g=yf$. Then $g\not\in \SR^1
(\R^2)$. Indeed, the partial derivative of $g$ along the $x$-axis is equal to
$${\partial}_{ x}g(x,y)=\dfrac{2xy^6}{(x^2+y^4)^2}$$ and 
$\lim_{t\rightarrow 0}{\partial}_{
  x}g(t^2,t)=\dfrac{1}{2}\not= 0={\partial}_{ x}g(0,0) $.
\item Take the same function $g$ as in the second example, so that $\pol(g)=\{ (0,0)\}$ and $g\not\in\SR^1(\R^2)$. However, 
$g(x,y)=o(\sqrt{x^2+y^2})$, so that $g$ is a differentiable function at the
origin.
\end{enumerate}
\end{ex}

In the following, we determine a subring of the ring of regulous
functions on $\R^2$ for which Questions \ref{ques} admit a positive answer. In the reminder of the paper, we investigate the properties of this ring.

\subsection{Regular functions after one blowing-up}

By \cite{FHMM}, we know that the regulous functions correspond exactly
to the functions which become regular after a sequence of blowings-up
along smooth centers. For regulous functions on the plane, only
blowings-up along a finite number of points are necessary, and we can easily
count the minimal number of steps necessary.
This is why, in the following, we focus on the study of
regulous functions on the plane.

Note that the denominator of a regulous function in the plane admits only a finite number of zeroes (\cite[Cor. 3.7]{FHMM}). In particular such a denominator has constant sign on $\R^2$.

In order to estimate the complexity of a sequence of blowings-up, we recall the definition of an infinitely near point.
Let $\pi:M\to \R^2$ be a successive composition of blowings-up along a point. For $n\in \mathbb N$, an infinitely near point $a_n$ of order $n$ on $M$ is given by a sequence of points $a_0, \ldots, a_n$ on $M_0=\R^2, M_1, \ldots ,M_n=M$ such that $M_i$ is the blowing-up of $M_{i-1}$ at $a_{i-1}$ and $a_i$ is a point of $M_i$ with image $a_{i-1}$. We define the number of stages of $\pi$  to be the maximal order of infinitesimal points in $M$. In particular, if the number of stages is zero, then $\pi$ is the identity.

\begin{defn} 
Let $f\in\SR^0 (\R^2)$ and let $l\in\N$. We say that $f$ is
regular after $l$ blowings-up if there exist a 
$l$-stages composition $\pi:M\to \R^2$ of successive blowings-up along points such that $f\circ \pi$ is regular.

For $k\in\N\cup\{\infty\}$, we denote by $\SR^{k}_{[l]}(\R^2)$ the set
of $k$-regulous functions $f$ which are regular after at most $l$
blowings-up.
\end{defn}

\begin{prop} For $k$ and $l$ in $\N$, the set $\SR^{k}_{[l]}(\R^2)$ is a ring.
\end{prop}

\begin{proof}
The number of stages necessary to regularize the product (or the sum) of two elements of $\SR^{k}(\R^2)$ is bounded by the maximum of the numbers of stages of both elements.
\end{proof}

\begin{ex}\begin{flushleft}\end{flushleft}
\begin{enumerate}
\item The regulous function $f$ given by $f(x,y)=\dfrac{x^3}{x^2+y^2}$ is regular after the blowing-up of $\R^2$ at the origin. In charts, we obtain $f(u,uv)=\dfrac{u}{1+v^2} $ and $f(uv,v)= \dfrac{u^3}{u^2+1}$ which are regular functions. Therefore $f\in \SR^{k}_{[1]}(\R^2)$.
\item Similarly, the regulous function $g$ given by $$g(x,y)=\dfrac{x^3(x-1)^3}{(x^2+y^2)((x-1)^2+y^2)}$$ is regular after a one stage blowing-up. Actually $g$ becomes regular after the blowing-up of the two points $(0,0)$ and $(1,0)$ in $\R^2$.
\item However the regulous functions $h$ defined by $h(x,y)=\dfrac{x^3}{x^2+y^4}$ does not belong to $\SR^{k}_{[1]}(\R^2)$, because the blowing-up of the unique indeterminacy point gives rise to a regulous function which is not regular. Actually, in one chart we obtain a regular function $h(u,uv)=\dfrac{u}{1+u^2v^4}$ whereas in the other chart $h_1(u,v)=h(uv,v)=\dfrac{u^3v}{u^2+v^2}$ is only regulous. The blowing-up of the origin in this second chart gives rise to a regular function since $h_1(r,rs)=\dfrac{r^2s}{1+s^2}$ and $h_1(rs,s)=\dfrac{r^3s^2}{r^2+1}$. As a consequence $h$ belongs to $\SR^{k}_{[2]}(\R^2)$.
\end{enumerate}
\end{ex}

Note that a regulous function $f$ belonging to $\mathcal R^k_{[l]}(\R^2)$ becomes automatically regular in the neighbourhood of an infinitely near point of order $l$.

\begin{lem}\label{lem-aux2} Let $l\in \N$ and let $\pi:M\to \R^2$ be a $l$-stages composition of blowings-up along points of $\R^2$. Let $a_l\in M$ be an infinitely near point of order $l$. For any $f\in \mathcal R^0_{[l]}(\R^2)$, the function $f\circ \pi$ is regular in a neighbourhood of $a_l$ in $M$.
\end{lem}

\begin{proof} Set $M_0=\R^2$ and denote by $\pi_i:M_i\to M_{i-1}$, with $i\in \{1,\ldots,l\}$, the sequence of blowings-up defining the $l$-stages composition. Let $a_i\in M_i$, with $i\in \{0,\ldots,l\}$ the sequence of points giving the infinitely near point $a_l\in M$. If $a_0$ is not in $\pol(f)$, the conclusion is immediate. The same holds true if one of the $a_i$'s, with $i \in \{1, \ldots, l-1\}$, is not in $\pol(f\circ \pi_1 \circ \cdots \circ \pi_i)$. Finally if all the $a_i$'s, for $i \in \{0,\ldots, l-1\}$, are poles of the successives pull back of $f$, then $a_l$ is necessarily a regular point of $f\circ \pi$ because $f$ belongs to $\mathcal R^0_{[l]}(\R^2)$.
\end{proof}

\begin{rem}  Let $k\in\N\cup\{ \infty\}$. 
Note that a function in $\SR^{k}_{[0]}(\R^2)$ is regular, so it belongs automatically to $\SR^{\infty}(\R^2)$. We have the following sequence of inclusions
$$\SR^{\infty}(\R^2)=\SR^{k}_{[0]}(\R^2)\subset \SR^{k}_{[1]}(\R^2)\subset \SR^{k}_{[2]}(\R^2)\subset\cdots\subset\SR^k(\R^2).$$
\end{rem}

\begin{defn}\begin{flushleft}\end{flushleft}
\label{definite}
\begin{enumerate}
\item Let $a\in\R^2$. 
For $p\in\R[x,y]$, one has the unique decomposition into the finite sum 
  $p=\sum_{i\geq 0} p_i$
  where $p_i$ is a polynomial which is either zero or lies in $\mathbf m^i_{a}\setminus \mathbf m^{i-1}_{a}$.  The smallest integer $l$ such that $p_l\not=0$ is the order of $p$ at $a$, namely $l=\ord_{a}(p)$. 
 This decomposition is called the homogeneous decomposition of $p$ at $a$, and the $p_i$'s are the homogeneous components.\par
  When $a=o$, the polynomial $p_i$ is either zero or homogeneous of degree $i$.
\item Let $q\in\R[x,y]$ be homogeneous. In case the zero set $\mathcal Z(q)$ of $q$ coincides with the origin, we say that the homogeneous polynomial $q$ is definite. In that case the sign of $q$ is constant on $\R^2$, and we may precise accordingly whether $q$ is positive or negative definite.
\item A polynomial $p\in\R[x,y]$ is said to be locally (positive) definite at a point $a\in \R^2$ if the smallest degree homogeneous component of $p$ at $a$ is (positive) definite.
\end{enumerate}
\end{defn}

Note that any divisor of a definite homogeneous polynomial is definite and more precisely
\begin{lem}\label{productdefinitepolynomials}
Let $q\in \R[x,y]$ be an homogeneous polynomial. If $q_1,\ldots,q_r$ are polynomials in $\R[x,y]$ such that $q=q_1q_2\ldots q_r$, then all $q_i$'s are homogeneous. Moreover $q$ is definite if and only if each $q_i$ is definite.
\end{lem}

\begin{proof}
If one of the $q_i$'s is not homogeneous, the product of the homogeneous components of lowest degree of the $q_i$'s is different from the product of the homogeneous components of highest degree of the $q_i$'s, in contradiction with the homogeneity of $q$. For the second part, it is sufficient to notice that the zero set of $q$ is the union of the zero sets of the $q_i$'s.
\end{proof}






Remind that if $f$ and $g$ are two functions from $\R^2$ to $\R$ and
$a\in\R^2$, we say that $f$ si negligeable (and we denote it by
$f\ll_{a} g$) in comparison with $g$ at the point $a$ if, for any $\epsilon>0$ there is one neighbourhood $U$ of $a$ such that for any $x\in U$, $|f(x)|\leq \epsilon |g(x)|$.\par
We say furthermore that $f$ is equivalent to $g$ at the point $a$ (and we denote it by $f\sim_{a} g$)
if $(f-g)\ll_a g$.

\par

Next lemma collects information about the order of the denominator and
numerator of a rational function. For its proof, it will be very
convenient to use of polar coordinates $(\rho,\theta)$ at the origin;
namely $x=\rho \cos\theta, y=\rho \cos\theta$, $\rho>0$,
$\theta\in[0,2\pi[$. 

\begin{lem}
\label{lem-polaire}
Let $f\in \R(\R^2)$ be a non-zero rational function such that $o=(0,0)\in \pol(f)$. Write
$f=\dfrac{p}{q}$ on $\dom (f)$ with $p,q\in\R[x,y]$ such that
$q$ doesn't vanish on $\dom (f)$. Let $p_n$ (resp. $q_m$) be the
homogeneous part of smallest degree of $p$ (resp. $q$), hence
$n=\ord_o(p)$ and $m=\ord_0(q)$.
\begin{enumerate}
\item \label{ordreindet}
Assume $f\in \SR^0(\R^2)$. 
 Then $n\geq m$ and $m$ is even. Moreover $f(o)=0$ if and only if $n>m$.
\item \label{lemhom}
Assume $q_m$ is definite. Then
\begin{enumerate}
\item
\begin{itemize}
\item if $n=m$ then $f$ is locally bounded at $o$.
\end{itemize}
\item
\begin{itemize}
\item $q$ is equivalent to $q_m$ at the origin.
\item $p\ll_o q$ if and only if $n>m$.
\item $p\sim_oq$ if and only if $n=m$ and $p_m=q_m$.
\end{itemize}
\end{enumerate}
\end{enumerate}
\end{lem}

\begin{proof}
\begin{enumerate}
\item  The order $m$ of $q$ at $o$ is necessary even since $o$ is an isolated zero of $q$
\cite[Prop. 3.5]{FHMM}. \par
Writing $f$ in polar coordinates, we get 
$$f(\rho\cos\theta,\rho\sin
\theta)=\rho^{n-m}\dfrac{p_n(\cos\theta,\sin\theta)+\rho\,
  \tilde{p}(\rho\cos\theta,\rho\sin 
\theta)}{q_m(\cos\theta,\sin\theta)+\rho\, \tilde{q}(\rho\cos\theta,\rho\sin
\theta)}$$
with $\tilde{p},\tilde{q}$ some real polynomials in two variables.

The homogeneous polynomial $p_n\times q_m$ is not identically zero on $\R^2$,
and hence there exists $\theta_0\in \R$ such that $q_m(\cos\theta_0,\cos\theta_0)\not=0$ and $p_n(\cos\theta_0,\cos\theta_0)\not=0$. When $\rho$ tends to zero with $\theta=\theta_0$, the continuity of $f$ at the origin implies $n\geq m$. Moreover we get $f(o)=0$ if $n>m$.

In case $n=m$ and $f(o)=0$, then the quotient 
$\dfrac{p_n(\cos\theta,\sin\theta)}{q_m(\cos\theta,\sin\theta)}$ vanishes 
for infinitely many $\theta\in\R$, which would say that $p_n=0$.

\item To show the first point, again let us write $q=q_m+\tilde{q}$ where $\deg \tilde{q}>\deg q_m=m$ and  

$$q(\rho\cos\theta,\rho\sin
\theta)=\rho^{m}\left(q_m(\cos\theta,\sin\theta)+\rho\,\tilde{q}(\rho\cos\theta,\rho\sin \theta)\right)$$
$$=\rho^{m}q_m(\cos\theta,\sin\theta)\left(1+\dfrac{\rho\,\tilde{q}(\rho\cos\theta,\rho\sin \theta)}{q_m(\cos\theta,\sin\theta)}\right)$$

which proves the first assertion since $|q_m(\cos\theta,\sin\theta)|$ does not vanish and hence is bounded from below.
\par
To prove the last point, it suffices to say that, by assumption, for any $\epsilon>0$, there is a neighbourhood $U$ of the origin where $|p-q_m|<\epsilon |q_m|$ and hence $\left |\frac{p-q_m}{q_m}\right |<\epsilon$. This shows that $\ord_o(p-q_m)>\ord_o (q_m)$ and hence $p_m=q_m$.

\end{enumerate}
\end{proof}

\begin{rem}
\begin{enumerate}
\item Point 2b) of Lemma \ref{lem-polaire} is no more true without the condition of definiteness of $q_m$ : for instance $x^2+y^4$ is not equivalent to $x^2$ at the origin.
\item The results of the lemma can be extended to $\R^n$ using spherical coordinates in place of polar coordinates.
\item One may replace the origin $o$ by any other point of the plane.
\end{enumerate}
\end{rem}

Let us state now one key tool for the following. 

\begin{thm}\label{equivfctregapres1eclat} Let $f\in \SR^0(\R^2)$. Write
$f=\frac{p}{q}$ on $\dom (f)$ with $p$ and $q$ coprime in $\R[x,y]$, $q$ positive on $\R^2$ and such that
$q$ doesn't vanish on $\dom (f)$. Then $f$ belongs to
$\SR^0_{[1]}(\R^2)$ if and only if $q$ is locally positive definite at any $a\in \pol(f)$.
\end{thm}

\begin{proof} 
As $\pol (f)$ consists of a finite number of points, the proof is local at these points. We discuss in the following the case of the origin of $\R^2$.

Changing $f$ by $f-f(o)$ (they keep the same denominator, the new
numerator and the denominator are still coprime), we may assume
$f(o)=0$. Put $n=\ord_o p$, $m=\ord_o q$, and denote by $p_n$ and
$q_m$ respectively the homogeneous part of 
smallest degree of $p$ and $q$ at the origin.
From Lemma \ref{lem-polaire}.\ref{ordreindet} we get that $n>m$ and $m$ is even.

Assume $f$ is regular after one blowing-up at the origin. 
Let
$\pi:M\rightarrow \R^2$ is the blowing-up at the origin. We
cover $M$ by two affine subsets $U_1,U_2$ such that 
in the first (resp. second) open subset, $\pi$ is given by 
$(u,v)\mapsto (u,uv)=(x,y)$ (resp. $(u,v)\mapsto (uv,v)$).
On $U_1$, the function $f_1=f\circ\pi_{|U_1}$ is 


$$f_1(u,v)=
u^{n-m}\dfrac{p_n(1,v)+u\tilde{p}(u,v)}{q_m(1,v)+u\tilde{q}(u,v)},$$
where $\tilde{p}$ and $\tilde{q}$ are polynomials in $u$ and $v$.

We claim that the polynomials in $u$ and $v$ namely
$\dfrac{p(u,uv)}{u^n}$ and $\dfrac{q(u,uv)}{u^m}$ are coprime as says the following:\\

\vskip 0.1cm\noindent
{\bf Lemma}\\
{\it Let $p(x,y)$ and $q(x,y)$ be two coprime polynomials in $\R[x,y]$. 
Then, $p(u,uv)$ et $q(u,uv)$ are coprime polynomials in $\R[u,v][u^{-1}]$. } 
\vskip 0.1cm

\begin{proof}
We write $p(u,uv)=r(u,v)s(u,v)$
and 
$q(u,uv)=r(u,v)t(u,v)$ where $r,s,t$ are in $\R[u,v][u^{-1}]$.
Then,
$p(x,y)=r(x,y/x)s(x,y/x)$ and $q(x,y)=r(x,y/x)t(x,y/x)$ in $\R[x,y][x^{-1}]$.
Hence, there are integers $i,j,l$ such that the polynomial $x^ir(x,y/x)$ is a common factor to $x^jp(x,y)$
and $x^lq(x,y)$ in $\R[x,y]$. Then $x^ir(x,y/x)$ is invertible in $\R[x,y][x^{-1}]$, namely $r(u,v)$ is invertible in 
$\R[u,v][u^{-1}]$.
\end{proof}

Since $f_1$ is regular on a neighborhood of the exceptional line $E$
with equation $u=0$
and since the order of vanishing of $f_1$ on the exceptional line is
$n-m$, we can write $f_1=u^{n-m}g_1$ with $g_1=\dfrac{p_n(1,v)+u.\tilde{p}(u,v)}{q_m(1,v)+u.\tilde{q}(u,v)}$ regular on a
neighborhood of the exceptional line $E$. Recall, by using the previous lemma, that the polynomials 
$p_n(1,v)+u\tilde{p}(u,v)$ and $q_m(1,v)+u\tilde{q}(u,v)$ are
coprime. If $q_m(1,v)$ had a real root denoted by $v_0$ then $(0,v_0)$
is a pole of $f'_1$, a contradiction. Thus $q_m(1,v)$ has no real root.\\
We proceed likewise with $f_2$, given by the second blowing up, to show that $q_m(u,1)$ has no real root. Hence
$q_m(x,y)$ vanishes only at the origin.\\

Conversely, assume now that $q_m$ is definite. Then, $q_m(1,v)$ and $q_m(u,1)$ do not
have real roots. It means that, using the above notations, the
functions $f_i=f\circ\pi_{|U_i}$ are regular on the exceptional
divisor and thus also on a neighborhood of the exceptional divisor.
\end{proof}

\subsection{Partial derivatives of regulous functions after one blowing-up}

The following proposition shows that the property of being regular
after one blowing-up passes to the partial derivative functions.
\begin{prop}
\label{derpart1reg}
Let $k$ be an integer and let $f\in \SR^{k+1}_{[1]}(\R^2)$. Then for any $v\in \R^2$, the partial derivative $\partial_v f$ belongs to $\SR^{k}_{[1]}(\R^2)$.
\end{prop}

\begin{proof} By Corollary \ref{equivalence1}, we get that $\partial_v
  f\in\SR^{k}(\R^2)$. So, we are left to prove that $\partial_v f $ is
  regular after one blowing up. Suppose $o\in
  \pol(f)$ and $o\in \pol(\partial_v f)$ and also that 
$f=\dfrac{p}{q}$ on $\dom (f)$ with $p$ and $q$ coprime in $\R[x,y]$, and
$q$ doesn't vanish on $\dom (f)$. By Theorem
\ref{equivfctregapres1eclat}, we know that $q_m$ is definite. On $\dom (f)$,
the rational function $\partial_v f=\dfrac{(\partial_v p)q-(\partial_v
  q)p}{q^2}$ is regular and we apply lemma \ref{productdefinitepolynomials}
  to conclude the proof.
\end{proof}


We show that we can answer by the affirmative to Question \ref{ques}.(1) in the
case of regular functions after one blowing-up.
We first state the elementary 

\begin{lem}\label{orderderivative}
Let $a\in \R^2$ and $f\in\R(x,y)$ be a rational function which admits a derivative at $a$ in the direction $v\in\R^2$. Then,
$$\ord_{a}\partial_vf\geq \ord_{a}f-1.$$
\end{lem}

\begin{proof}

\end{proof}

\begin{prop}
\label{locbound}
Let $f\in\SR^{0}_{[1]}(\R^2)$. The partial derivative
functions of $f$ are locally bounded at the points of $\pol(f)$.
\end{prop}

\begin{proof} We are going to apply Lemma \ref{lem-polaire}.\ref{lemhom} Let $v\in \R^2$. It is sufficient to work locally at one indetermination point, which we suppose to be the origin. Moreover we may assume $f(o)=0$ 
since $\partial_v
f=\partial (f-f(o))$. 
Then
$f=\dfrac{p}{q}$ on $\dom (f)$ with $p$ and $q$ coprime in $\R[x,y]$, and
$q$ doesn't vanish on $\dom (f)$. We can write 
$$p=p_n+p_{n+1}+\cdots$$
$$q=q_m+q_{m+1}+\cdots$$
with $p_i, q_i$ homogeneous polynomials of degree $i$ in $\R[x,y]$ and 
$n=\ord_o(p)$, $m=\ord_0(q)$. By Lemma \ref{lem-polaire}.\ref{ordreindet} and Theorem
\ref{equivfctregapres1eclat}, $n>m$ and $\Z(q_m)=\{o\}$. On $\dom(f)$, we
have $$\partial_v f=\dfrac{(\partial_v p)q-(\partial_v
  q)p}{q^2},$$
and therefore $\partial_v f$ is locally bounded at $o$ by Lemma \ref{lem-polaire}.\ref{lemhom},
 Lemma \ref{orderderivative} and Lemma \ref{productdefinitepolynomials}.
\end{proof}



We show that we can answer by the affirmative to Question \ref{ques}.(2) in the
case of regular functions after one blowing-up.
\begin{thm}
\label{question2}
Let $k$ be an integer and let $f\in\SR^{k}_{[1]}(\R^2)$.  Let $h\in
\SR^{k+1}_{[1]}(\R^2)$ such that $\pol (f)\subseteq \Z(h)$, then
$$hf\in \SR^{k+1}_{[1]}(\R^2).$$
\end{thm}

\begin{proof} By stability under product, the function $hf$ belongs to $\SR^{0}_{[1]}(\R^2)$. As a consequence, it is sufficient to prove
  that $hf\in\SR^{k+1}(\R^2)$.\\
We proceed by induction on $k$. Since being $C^k$ is
a local property, we assume in the following that $\pol (f)=\{ o\}$ and $h(o)=0$.

For the case $k=0$, note that on $\R^2\setminus \{ o\}$, we have $\partial_v (hf)=f\partial_v
h+h\partial_v f$ for $v\in \R^2$. The product $f\partial_v h$ belongs to $\SR^{0}(\R^2)$ since
$f\in\SR^{0}(\R^2)$ and $h\in\SR^1 (\R^2)$. By Proposition
\ref{locbound}, $\partial_v f$ is locally bounded at $o$, hence
$h\partial_v f$ can be extended continuously by $0$ at $o$. As a consequence we get $hf\in\SR^{1}(\R^2)$ by Proposition
\ref{equivalence3}.

Concerning heredity, let $f\in\SR^{k+1}(\R^2)$ and $h\in
\SR^{k+2}(\R^2)$ such that $\{o\}=\pol (f)\subseteq \Z(h)$. We have to show that any partial derivative of $hf$ is of class $C^{k+1}$. On $\R^2\setminus \{ o\}$, we have $\partial_v (hf)=f\partial_v
h+h\partial_v f$ for $v\in \R^2$. Since $\partial_v h\in\SR^{k+1}(\R^2)$, we only have to prove that $h\partial_v f$ belongs to $\SR^{k+1}(\R^2)$. This is true by the induction hypothesis since
$\pol (\partial_v f)\subseteq \pol (f)$.
\end{proof}


We show that we can answer by the affirmative to Question \ref{ques}.(3) in the
case of regular functions after one blowing-up.
\begin{thm}
\label{question3}
Let $k$ be an integer and $f\in\SR^{0}_{[1]}(\R^2)$. Then $f\in
\SR^{k}_{[1]}(\R^2)$ if and only if $f$ has a polynomial
  expansion of
  order $k$ at any point of $\R^2$.
\end{thm}

\begin{proof} 
The direct implication is given by Taylor expansion.
For the converse implication, we proceed by induction on $k$. For $k=0$
the proof is trivial.

Assume $f$ has a polynomial
  expansion of
  order $k>0$ at any point of $\R^2$ with $k\geq 1$. By the induction
  hypothesis, the function $f$ belongs to $\SR^{k-1}_{[1]}(\R^2)$, namely for all $j$  in $\{1,\ldots,k-1\}$, the
  $j$-th partial derivative functions of $f$ can be extended continuously at $o$.
Since $f$ is already regular after a one stage blowing-up, we only have to prove that $f$ is of class $C^k$. We may also assume $\pol (f)=\{o\}$ and
$f=o(\sqrt{x^2+y^2}^k)$ at $o$ (by substracting to $f$ its polynomial
approximation of degree $k$ at $o$). Then
$f=\dfrac{p}{q}$ on $\dom (f)$ with $p$ and $q$ coprime in $\R[x,y]$, and
$q$ doesn't vanish on $\dom (f)$. We can write 
$$p=p_n+p_{n+1}+\cdots$$
$$q=q_m+q_{m+1}+\cdots$$
with $p_i, q_i$ homogeneous polynomials of degree $i$ in $\R[x,y]$ and 
$n=\ord_o(p)$, $m=\ord_0(q)$. We are going to prove that $n\geq m+k+1.$ Actually, by Lemma \ref{lem-polaire}.\ref{ordreindet} and Theorem
\ref{equivfctregapres1eclat}, we have already $n>m$ and $q_m$ is definite.
In polar coordinates, we get 
$$f(\rho\cos\theta,\rho\sin
\theta)=\rho^{n-m}\dfrac{p_n(\cos\theta,\sin\theta)+\rho\,
  \tilde{p}(\rho\cos\theta,\rho\sin 
\theta)}{q_m(\cos\theta,\sin\theta)+\rho\, \tilde{q}(\rho\cos\theta,\rho\sin
\theta)}$$
with $\tilde{p},\tilde{q}$ some polynomials in two variables. By hypothesis, 
$$\dfrac{f(\rho\cos\theta,\rho\sin
\theta)}{\rho^k}=\rho^{n-m-k}\dfrac{p_n(\cos\theta,\sin\theta)+\rho\,
  \tilde{p}(\rho\cos\theta,\rho\sin 
\theta)}{q_m(\cos\theta,\sin\theta)+\rho\, \tilde{q}(\rho\cos\theta,\rho\sin
\theta)}$$ is a $o(1)$ at $(0,0)$ and thus $$n\geq m+k+1.$$
Let $\partial^k_{v_1\cdots v_k} f$ be the partial derivative of order $k$ of
$f$ along the vectors $v_1,\ldots,v_k\in \R^2$. Then $\partial^k_{v_1\cdots v_k} f$ is defined on $\dom (f)$, and we want to prove that it admits a continuous extension to $\R^2$. Then $\partial^k_{v_1\cdots v_k} f$ has the form $\partial^k_{v_1\cdots v_k} f=\dfrac{r}{q^{2^k}}$, with
some $r\in\R[x,y]$. By Lemma \ref{orderderivative}, the order of $r$ is greater than or equal to $n-m+2^km-k$.

Combining this inequality with $n\geq m+k+1$, we obtain $\ord_o r>2^km$.
By Lemma \ref{lem-polaire}.\ref{lemhom}, the rational function $\partial^k_{v_1\cdots v_k} f$ can be
extended continuously by  $0$ at $o$, which is sufficient to prove that $f$ is $C^k$ by Theorem \ref{thm-top}.
\end{proof}

We give an application of Theorem \ref{question3} to the study of flatness of regulous functions. In particular we are able to give an explicit bound in Proposition \ref{kflat2} in the case of regulous functions regular after one stage of blowings-up.

\begin{cor}\label{kflat[1]} Let $k\in\N^*$ and $f\in\SR^{0}_{[1]}(\R^2)$. For $m\geq 2k$, the function $f^m$ is $k$-flat.
\end{cor}

\begin{proof} Following the proof of Proposition \ref{kflat2} in the case $k=1$, we see that $f^m$ is of class $C^1$ on $\Z(f)$ and $1$-flat for any $m\geq 2$, because the first partial derivatives of $f$ are locally bounded by Proposition \ref{locbound}.

Since higher derivatives are not necessarily locally bounded, we
cannot obtain such a bound for $k>1$ by following the proof of
Proposition \ref{kflat2}. Note that, by successive application of
Proposition \ref{kflat2}, the function $f^{2^k}$ is $k$-flat.
However, using Theorem \ref{thm-top}, we have that $f^2$ admits at the
order one the zero polynomial expansion  at any point of $\mathcal
Z(f)$, namely the order of $f^2$ at any point of $\Z(f)$ is greater
than one. By product, the function $(f^{2})^k$ admits also the zero
polynomial expansion at any point of $\mathcal Z(f)$, but this time at the order $k$. In particular, by Theorem \ref{question3} $f^m$ is of class $C^k$ at any point of $\mathcal Z(f)$ for $m\geq 2k$, and is $k$-flat.
\end{proof}

\section{Comparison of topologies}
\label{topo}


For an integer $k$, the $k$-regulous topology of $\R^n$ is defined to be the topology whose closed subsets are generated by the zero sets of regulous functions in $\SR^k (\R^n)$. 
Although the $k'$-regulous topology is a priori finer than the $k$-regulous topology when $k'<k$, it has been proved in   \cite{FHMM} that in fact they are the same.
Hence, it is not necessary to specify the integer $k$ to define the regulous topology on $\R^n$.
It raises naturally the question to know whether it is the same for the topology generated by regulous functions in  
$\SR^k_{[l]} (\R^2)$ for some integers $k,l$. \par
On $\R^2$, we define the $k_{[l]}$-regulous
topology (or simply the $[l]$-regulous topology when $k=0$) to be the
topology generated by zero sets of functions in $\SR^{k}_{[l]}(\R^2)$. 
Since the regulous topology is noetherian (\cite{FHMM}), one deduces :
\begin{prop}
The $k_{[l]}$-regulous topology is noetherian. Moreover, for any $k_{[l]}$-regulous closed set $F$, there
exists $f\in\SR^{k}_{[l]}(\R^2)$ such that $\Z (f)=F$. 
\end{prop}
\begin{proof}
The noetherianity is given from the fact that the regulous topology is
finer than the $k_{[l]}$-regulous topology. Then, any
$k_{[l]}$-regulous closed subset $F$ can be written 
$F=\Z(f_1)\cap \cdots\cap \Z(f_r)$ for given $f_1,\ldots,f_r$ in
$\SR^k_{[l]}(\R^2)$. It suffices to take $f=f_1^2+\cdots+f_r^2$.
\end{proof}

One may wonder how to compare $k_{[l]}$-regulous topologies when $k$ and $l$ vary. We study this question in the next subsections, respectively when the number $l$ of stages of blowings-up vary with fixed regularity $k$ and next 
when the regularities $k$ vary with a fixed number $l$ of stages of blowings-up.

We begin with the case when the number of stages of blowings-up vary.
It is not so difficult to see that the regulous
topology is strictly finer than the ${[1]}$-regulous topology. \par

Let us consider the real plane curve $V=\Z (y^2-x^4(x-1))$. The curve $V$ has two connected components consisting of a smooth $1$-dimensional branch $F$ and the origin $o$ as an isolated point. The branch $F=V\setminus o$ is a ${[2]}$-closed subset since the function $f$ defined by 
$$f(x,y)=1-\dfrac{x^5}{y^2+x^4}=\dfrac{y^2+x^4-x^5}{y^2+x^4},$$
which belongs to $\SR_{[2]}^0(\R^2)$, satisfies $\mathcal Z(f)=F$.

Let us show that $F$ is not a ${[1]}$-closed subset. By the contrary, assume this is the case, namely that $F=\Z(g)$ with $g\in\SR_{[1]}^0(\R^2)$. 
By the regulous Nullstellensatz \cite{FHMM}, there exist a function
$h\in\SR^0(\R^2)$ and a positive integer $n$ such that $g^n=hf$, so
that $hf$ belongs to $\SR_{[1]}^0(\R^2)$. Since $f$ has a unique pole
at the origin, the function h is therefore regular after a one stage
blowing-up outside the origin by Theorem
\ref{equivfctregapres1eclat}. Let us focus now on the origin. Set
$h=\dfrac{r}{s}$ with $r,s\in \R[x,y]$ coprime. Note that $h(o)\neq 0$
so that the order of $r$ and $s$ coincide by Lemma
\ref{lem-polaire}.(\ref{ordreindet}). We denote it by $m$. By Theorem
\ref{equivfctregapres1eclat} again, the homogeneous part of smallest
degree of the denominator of $hf$ at the origin is definite, therefore
$y^2+x^4$ divides $r$ (cf. Lemma
\ref{productdefinitepolynomials}). Note that $y^2+x^4-x^5$ cannot
divide $s$ because the zero set of $s$ is finite, and therefore the
homogeneous part $s_m$ of $s$ of smallest degree is definite. As a
consequence the homogeneous part $r_m$ of $r$ of smallest degree
satisfies $r_m=h(o)s_m$ by Lemma \ref{lem-polaire}.(\ref{lemhom}). As
a consequence $r_m$ should be definite, in contradiction with the
divisibility of $r$ by $y^2+x^4$ (cf. Lemma
\ref{productdefinitepolynomials}).

\vskip 5mm

One may generalise this example as follows.
Let $F$
be a one-dimensional closed, irreducible regulous subset of $\R^2$. Let $V$ denote the Zariski closure of $F$. By \cite{FHMM}, the Euclidean closure
$\overline{V_{reg}}^{Eucl}$ of the regular part of $V$ coincides with $F$, and $V\setminus F$ consists of a finite number of
points which are isolated in $V$. We describe below the belonging of $F$ to the ${[1]}$-regulous topology in terms of the behaviour of $F$ under the blowing-up along these isolated points.
\begin{prop}
\label{critere} Let $F$
be a closed, irreducible ${[1]}$-regulous subset of $\R^2$. Let $V$ denote
the Zariski closure of $F$. Assume $F$ is not equal to $V$, and choose  $a\in V\setminus F$. Let $\pi_a:M_a\to \R^2$ be the blowing-up of $\R^2$ at
$a$ and denote by $\widetilde{V}_a$ the strict transform of $V$ by
$\pi_a$ and by $E_a\subset M_a$ the exceptional divisor. Then
$$\widetilde{V}_a\cap E_a=\emptyset.$$
\end{prop}

\begin{proof}
We may assume that $V\setminus F=\{a\}$, working on a Zariski open subset of 
$\R^2$ if necessary. There exists $f\in\SR^{0}_{[1]}(\R^2)$ such that $\Z (f)=F$ since $F$ is a ${[1]}$-regulous closed set.
Write $f=\frac{p}{q}$ with $p,q\in\R[x,y]$ coprime and
$q$ doesn't vanish on $\dom (f)$. By Lemma \ref{lemme1}, we have 
$F\subseteq \Z(p)$ and thus $V\subseteq \Z(p)$. Hence $p$ vanishes at
$a$ and thus $q$ also vanishes at $a$ since $f$ doesn't vanish at $a$.
As a consequence $a\in \pol(f)$. By \cite{FHMM}, $\pol (f)$ consists of a finite
number of points. Working on a Zariski open subset of $\R^2$ if
necessary, we may assume that $\pol (f)=\{ a\}$. Since
$f\in\SR^{0}_{[1]}(\R^2)$ and $\pol (f)=\{ a\}$, the composition $f\circ \pi_a$ is regular and therefore
$\pi_a^{-1}(F)$ is a one-dimensional Zariski closed subset of
$\widetilde{V}_a$. Moreover $\widetilde{V}_a$ is irreducible because so is $V$, so necessarily 
$\pi_a^{-1}(F)=\widetilde{V}_a$. The conclusion follows since $a\notin F$. 
\end{proof}

\vskip 1cm 
Generalizing this example, one gets the following. Let again $F$ be a closed, irreducible one-dimensional regulous subset of $\R^2$. The Zariski closure $V$ of $F$ is an irreducible algebraic set, with possibly some isolated points $\{a_1,\ldots,a_n\}=V\setminus F$. Let $\pi_1:M_1\to \R^2$ denote the blowing-up of $\R^2$ along $\{a_1,\ldots,a_n\}$. Then the strict transform $V_1$ of $V$ in $M_1$ is again an irreducible real algebraic set with a one-dimensional part equal to $\pi_1^{-1}(F)$ and possibly some isolated points $\{a^1_1,\ldots,a^1_{n_1}\}$. One may blow-up $M_1$ along $\{a^1_1,\ldots,a^1_{n_1}\}$ to define similarly $\pi_2 :M_2\to M_1$, the strict transform $V_2$ and possibly some isolated points in $V_2\setminus (\pi_1\circ \pi_2)^{-1}(F)$. Continuing the procedure, we define recursively a sequence of blowings-up $\pi_i:M_i\to M_{i-1}$ and the strict transforms $V_i\subset M_i$ of $V$ (with $\pi_i$ the identity map in case $V_{i-1}$ has no isolated point). For an integer $l\in \N$, we denote by $\pi_{[l]}:M_l\to \R^2$ the composition $\pi_1\circ \cdots \circ \pi_l$, called the $[l]$-multiblowing-up of V along its set of isolated points.

\begin{thm}\label{thm-multi}
Let $F$ be a closed, irreducible one-dimensional regulous subset of $\R^2$, with Zariski closure $V$. Assume $F$ is a $[l]$-regulous set. Then the strict transform $V_l$ of $V$ by the $[l]$-multiblowing-up $\pi_{[l]}$ of V along $V\setminus F$ has no isolated points. In particular, $\pi_{[l]}^{-1}(F)=V_l$ is an algebraic set.
\end{thm}

Before entering into the details of the proof, we state an auxiliary result which deals with regulous fonctions on a non-singular real algebraic set (cf. \cite{FHMM}).

\begin{lem}\label{lem-aux1}
Let $X\subset \R^n$ be a non-singular real algebraic set and $f\in \mathcal R^0(X)$ be a regulous fonction on $X$. Let $V$ be the Zariski closure in $X$ of the zero set $\mathcal Z(f)$ of $f$. Then $V\setminus \mathcal Z(f)$ is contained in $\pol(f)$.
\end{lem}

\begin{proof} By \cite{Ko}, $f$ is the restriction to $X$ of a regulous
  function on $\R^n$. Thus we may write $f$ as the quotient $\frac{p}{q}$ of coprime polynomials $p,q\in \R[x_1,\ldots,x_n]$, we obtain that $V$ is necessarily contained in the zero set $\mathcal Z(p)$ of $p$. In particular $q$ must vanish at any point of $\mathcal Z(p)$ where $f$ does not vanish.
\end{proof}

\begin{proof}[Proof of Theorem \ref{thm-multi}]
There exists $f\in \mathcal R^0_{[l]}(\R^2)$ such that $\mathcal Z(f)=F$. By Lemma \ref{lem-aux1}, the set of poles of $f$ contained in $V\setminus F$ is exactly $V\setminus F$. Then $f\circ \pi_{[l]}$ is regular in a neighbourhood of $\pi_{[l]}^{-1}(V\setminus F)$ by Lemma \ref{lem-aux2}. 
As any isolated point of $V_l$ is a pole of $f\circ \pi_{[l]}$ by \ref{lem-aux1}, necessarily contained in $\pi_{[l]}^{-1}(V\setminus F)$, the regularity of $f\circ \pi_{[l]}$ in a neighbourhood of $\pi_{[l]}^{-1}(V\setminus F)$ induces the absence of isolated point in $V_l$.
As a consequence $\pi_{[l]}^{-1}(F)$ is equal to the one-dimensional part of $V_l$, so it is equal to $V_l$ because $V_l$ has no isolated point.
\end{proof}

\begin{cor} For any integers $l'<l$, the $k_{[l']}$-regulous topology is strictly finer than the $k_{[l]}$-regulous topology.
\end{cor}

\begin{proof}
Let $m\in \N$ be an odd integer and consider the real plane curve $V=\Z (y^2-x^{2l}(x^{m}-1))$. The curve $V$ has two connected components consisting of a smooth $1$-dimensional branch $F$ and the origin $o$ as an isolated point. For $m$ big enough, the branch $F=V\setminus o$ is a $k_{[l]}$-closed subset. Actually, $F$ is the zero set of the function $f$ defined by 
$$f(x,y)=1-\dfrac{x^{2l+m}}{y^2+x^{2l}},$$
whick is of class $C^k$ for $m$ big enough by \L ojasiewicz property \cite[Lem. 5.1]{FHMM}. Moreover $f$ belongs to $\mathcal R^k_{[l]}(\R^2)$ because after performing $l$ successive blowings-up of the origin, in the only relevant chart the function $f$ becomes the regular function
$$f(u, u^lv)=1-\dfrac{u^{m}}{v^2+1}.$$
Finally $F$ cannot be a $k_{[l-1]}$-closed set by Theorem \ref{thm-multi}, because after $l-1$ successive blowings-up of the origin in the chart corresponding to $(u,v)\mapsto (u,uv)$, the equation of the strict transform $V_{l-1}$ of $V$ is $v^2-u^{2}(u^{m}-1)=0$ and thus $V_{l-1}$ still admits an isolated point at the origin.
\end{proof}

Consider now the study of the topologies when the regularity varies.
So we fix a number of stages of blowings-up $l\in \mathbb N$, and compare the topologies when the regularity $k\in \N$ varies.

\begin{thm}
\label{egaltop}
Let $l\in\N^*$ and $k,k'\in\N^2$. The $k_{[l]}$-regulous
topology and the $k'_{[l]}$-regulous topology coincide on $\R^2$.
\end{thm}

\begin{proof}
It suffices to show that the $k_{[l]}$-regulous
topology coincides with the ${[l]}$-regulous topology.

Let $F$ be a ${[l]}$-regulous closed subset of
$\R^2$. By Noetherianity, there exists $f\in \SR^{0}_{[l]}(\R^2)$ such
that $F=\Z (f)$. By Proposition \ref{kflat2}, there exist an integer $N$ such that, replacing $f$ by $f^N$ if necessary, we
may assume that $f$ is already of class $C^k$ at any point of $\Z(f)$.

If $\pol(f)\subset  \Z(f)$, then $f\in \SR^{k}_{[l]}(\R^2)$ and 
$F$ is a closed $k_{[l]}$-regulous subset of
$\R^2$. 

Now assume there exists a point $a$ in $\pol(f)\setminus \Z(f)$. By Proposition
\ref{kflat2}, there exists an odd integer $N$ such that $(f-f(a))^N$
is of class $C^k$ at $a$ (we may even choose $N=2k+1$ by Corollary
\ref{kflat[1]} in the case $l=1$). Set $g=(f-f(a))^N+f(a)^N$. Then $g$ belongs to $
\SR^{0}_{[l]}(\R^2)$, and the zero sets $\Z(g)$ and $\Z(f)$ coincide because $N$ is odd. Moreover $\pol(g)\subseteq \pol(f)$ and $g$
is of class $C^k$ at any point of $\Z(g)\cup \{a\}$. We achieve the proof by repeating the
same method for any point in the finite set $\pol(f)$.
\end{proof}

\section{Real Algebra in the rings of regulous functions}

\subsection{Introduction}

In this section, we study the algebraic properties of the rings $\SR_{[l]}^{k}(\R^2)$
in comparison with those of $\SR^{k}(\R^2)$.\par
In \cite{FHMM}, the so-called \L ojasiewicz property (\cite[Lem. 5.1]{FHMM}) that we already mentioned is the key tool to prove the ``weak'' and ``strong'' Nullstellens\"atze in the ring $\SR^{k}(\R^n)$.\par
Unfortunately, the \L ojasiewicz property is no more true in the ring $\SR^{k}_{[1]}(\R^2)$. Indeed, the rational function
  $\frac{1}{x^2+y^4}$ is regular on $\R^2\setminus \Z(x)$ and thus is regular
  after one blowing-up on $\R^2\setminus \Z(x)$. But, for any $N\in\N$,
  $\frac{x^N}{x^2+y^4}$ extended by $0$ at the origin is not regular
  after the blowing-up of the origin.\par
Thus, our aim is first to give a new proof of the Nullstellensatz for
$\SR^{k}(\R^n)$ that doesn't use the \L ojasiewicz property. Then, we will be able to extend this proof to
the ring $\SR^{k}_{[l]}(\R^2)$ at least to get a {\it real}
Nullstellensatz, since the Nullstellensatz appears not to be valid in
$\SR^{k}_{[1]}(\R^n)$. We also prove the radical principality of the
ring $\SR^{k}_{[1]}(\R^2)$ which can be seen as a weak \L ojasiewicz property. \par
We complete the study of real algebraic properties of the ring $\SR^{k}_{[1]}(\R^2)$ by mentioning an Artin-Lang property which will be useful to get a Positivstellensatz and which will be also useful in the next section about Hilbert $17$-th problem.

\subsection{Semialgebraic subsets and Tarski-Seidenberg Theorem}

We recall that a semialgebraic subset of $\R^n$ is a subset of
$\R^n$ which is a boolean combination of subsets of the
form $\{x\in\R^n|\,\,p(x)>0\}$ with $p\in \R[x_1,\ldots,x_n]$. \par
It shall be noted that, since the regulous functions are
semialgebraic \cite[Pro. 3.1]{FHMM}, if we replace in the definition
polynomials by regulous functions, it does not create new sets. More precisely:

\begin{prop}
\label{kregsa=sa}
Let $k,n\in\N\times\N^*$. Any boolean combination of subsets of the
form $\{x\in\R^n|\,\,f(x)>0\}$, where $f\in \SR^{k}(\R^n)$, is a semialgebraic subset of $\R^n$. 
\end{prop}



As a by-product of Tarski-Seidenberg's Theorem, one may also mention the possibility of extending a regulous function to 
any real closed extension. If $f$ is a rational function in
$\R(x_1,\ldots,x_n)$ and $\R\rightarrow R$ a real closed fields
extension, then one may define 
(independently of the representation $f=\dfrac{p}{q}$) the function $f_R=\dfrac{p}{q}$ viewed as a rational function in $R(x_1,\ldots,x_n)$.\par 
One also easily mimics the definition of the ring $\SR^{k}_{[l]}(\R^2)$ to the ring $\SR^{k}_{[l]}(R^2)$ of $k$-regulous functions regular after $l$ blowing-ups defined over a real closed field $R$ in place of the field of usual real numbers $\R$. One has 

\begin{prop}\label{substitution}
Let $f\in \R(x_1,\ldots,x_n)$, $(k,l)\in(\N\cup\{\infty\})\times \N$ and $\R\rightarrow R$ a real closed fields extension.\par Then, 
$f\in \SR^{k}(\R^n)$ if and only if $f_R\in \SR^{k}(R^n)$. Likewise
$f\in \SR^{k}_{[l]}(\R^2)$ if and only if $f_R\in \SR^{k}_{[l]}(R^2)$.
\end{prop}
\begin{proof}
Let us see first that the $k$-regularity condition is semialgebraic.\par
Indeed, by Proposition \ref{equivalence3bis}, $f\in\SR^{k}(\R^n)$ if and only if
$\partial^if\in\SR^{0}(\R^n)$ for any derivative $\partial^{i}$
of order $i$ with $0\leq i\leq k$ ($\partial^0f=f$). It means that any rational function $\partial^kf$ seen
on $\dom(f)$ (a semialgebraic set)
can be extended
continuously to $\pol(f)$ (a semialgebraic set). 
We will see that  this kind of property is semialgebraic: Let $g=p/q$
with  $p,q\in \R [x_1,\ldots,x_n]$ and let $x\in\RR^n$. Then $g$ can
be extended continuously at $x$ may be written as :\par 
$$\begin{array}{l}
\exists l\,\forall \varepsilon>0,\,\exists \eta>0,\,\,
\forall y\,\left(q(y)\not=0\,\,{\rm
  and}\,\,||y-x||<\eta\right)\\
\Longrightarrow\,\,|p(y)-lq(y)|<\varepsilon
|q(y)|.
\end{array}$$ It is a first order formula in the language of ordered fields with
variables in $\R$. By \cite[Prop. 2.2.4]{BCR}, this property is semialgebraic.
Then, using  Tarski-Seidenberg's Theorem, we get that $f\in \SR^{k}(\R^n)$ if and only if $f_R\in \SR^{k}(R^n)$.\par
The condition to be regular after $l$ blowings-up at a finite number
of given points is also semialgebraic (given such a function, one has to write that some
polynomials, namely some denominators after the composition with the given blowings-up, do not vanish) and hence we also have
$f\in \SR^{k}_{[l]}(\R^2)$ if and only if $f_R\in \SR^{k}_{[l]}(R^2)$.
\end{proof}

\subsection{Real algebra and Artin-Lang property}
\label{RadicPrincip}

We begin with some preliminary settings about real algebra.
Let $A$ be a commutative ring containing $\QQ$.
An order $\alpha$ in $A$ is given by a prime ideal $p$ of $A$ and an ordering on the residual 
field $k(p)$ at $p$ or equivalently it is given by a morphism $\phi$ from $A$ to a real closed field $K$. The value $a(\alpha)$ of $a\in A$ at the ordering $\alpha$ is just $\phi(a)$. 
The set of orders of $A$ is called the real spectrum of $A$ and denoted by $\Sp_r A$. 
It is empty if and only if $-1$ is a sum of squares in $A$.  
One endows $\Sp_r A$ with a natural topology whose open subsets are generated by the sets $\{\alpha\in\Sp_rA\mid a(\alpha)>0\}$ where $a\in A$. For more details
on the real spectrum, the reader is referred to \cite{BCR}.\par
An ideal $I$ in $A$ is said to be real if whenever $\sum_{i=1} ^p a_i ^2$ is in $I$ then any $a_i$ is in $I$.
We introduce the real radical of the ideal $I$ denoted by $\RadRe{I}$ the set $$\{a\in A|\,\exists
m\in\N,\,\exists b_1,\ldots,b_p\in A,\,\,
a^{2m}+b_1^2+\cdots+b_p^2\in I\}.$$ By \cite[Lem. 4.1.5]{BCR},
$\RadRe{I}$ is the intersection of the real prime ideals
of $A$ containing $I$. We also have that
$I\subset\Rad{I}\subset\RadRe{I}$ with equality if and only if $I$ is
a real radical ideal.

Back to our ring $\SR^{k}(\R^n)$, let us formulate an elementary but essential substitution property for regulous functions : 

\begin{prop}\label{SubstitutionRegulue}
Let $(k,n)\in\N\times\N^*$. Let $\phi:\Sp_r\SR^{k}(\R^n)\rightarrow R$ be a ring homomorphism where $R$ is a real closed extension of $\R$. For any $f\in\SR^{k}(\R^n)$, one has
$$\phi(f)=f_R\left(\phi(x_1),\ldots,\phi(x_n)\right)$$
\end{prop}
\begin{proof}
One may assume, for simplicity, that $\phi(x_1)=\ldots=\phi(x_n)=0$. 
Let $f\in\SR^{k}(\R^n)$. Up to considering $f-f(o)$, we may assume that $f(o)=0$ and we have to show that $\phi(f)=0$.\par
We use the  \L ojaciewicz property in $\SR^{k}(\R^n)$ (cf \cite[Lemme 4.1]{FHMM}) to the functions $f$ and $x_1^2+\ldots+x_n^2$ which are such that  $\Z( x_1^2+\ldots+x_n^2)\subset\Z(f)$. It says that there exists an integer $N$ and a regulous function $g$ in $\SR^{k}(\R^n)$ such that 
$f^N=(x_1^2+\ldots+x_n^2)g$, an algebraic identity which implies $\phi(f)=0$.
\end{proof}

The previous proposition says that if a morphism from $\SR^{k}(\R^n)$ is the evaluation at a given point $x\in R^n$ in restriction to the polynomials, then it is evaluation at $x$ on any regulous functions of $\R^n$.\par  
Here is now our Artin-Lang property, fondamental to obtain algebraic identities involved in Positivstellens\"atze :

\begin{prop}\label{ArtinLang}
Let $(k,n)\in\N\times\N^*$. Let $f_1,\ldots,f_r$ in $\SR^{k}(\R^n)$ and set $$S=\{x\in\R^n\mid f_1(x)>0,\ldots,f_r(x)>0\},$$
and 
$$\widetilde{S}=\{\alpha\in\Sp_r\SR^{k}(\R^n)\mid f_1(\alpha)>0,\ldots,f_r(\alpha)>0\}.$$
Then,
$\widetilde{S}=\emptyset$ if and only if $S=\emptyset$.
\end{prop} 
\begin{proof}
Since $S$ is a semialgebraic subset of $\R^n$ then  
there exist $(l,m)\in\N^2$, some
polynomials $p_{i,j}$ and $q_i$, $1\leq i\leq l$, $1\leq j\leq m$ in $\R[x_1,\ldots,x_n]$ such that
$$S=\cup_{i=1}^l S_i$$ with $S_i=\{x\in\R^n\mid
q_i(x)=0,p_{i,1}(x)>0,\ldots, p_{i,m}(x)>0\}$ (see \cite{BCR}).
By Tarski-Seidenberg's Theorem
\cite[Prop. 5.1.4]{BCR}, $S=\emptyset$ if and only if for any real closed extension $R$ of $\R$ and
for any $i\in\{1,\ldots,l\}$, we have $$S_{i,R}=\{x\in R^n\mid
q_i(x)=0,p_{i,1}(x)>0,\ldots, p_{i,m}(x)>0\}=\emptyset.$$
Assume $S=\emptyset$ and $\widetilde{S}\not=\emptyset$. Let $\alpha\in \widetilde{S}$.
Then $\alpha$ corresponds to a morphism
$\phi:\SR^{k}(\R^n)\rightarrow R$ where $R$ is real closed extension
of $\R$. Let $y=(x_1(\alpha),\ldots,x_n(\alpha))=(\phi(x_1),\ldots,\phi(x_n))\in R^n$. Since
$\alpha\in \widetilde{S}$, there exists, by the substitution property
(Proposition \ref{SubstitutionRegulue}), an index $i$ such that $y\in
S_{i,R}$, a contradiction.\par  In other words, the substitution property gives the description of 
$\widetilde{S}$ as the inductive limit of all $S_R=\{x\in R^n\mid f_{1,R}(x)>0,\ldots,f_{r,R}(x)>0\}=\cup_{i=1}^l S_{i,R}$ where the union rides over all the real closed extensions $\R\rightarrow R$.  
\end{proof}

The same arguments can be applied to the ring $\SR^{k}_{[l]}(\R^2)$. First, we have a substitution property:

\begin{prop}\label{SubstitutionRegulueBlow}
Let $(k,l)\in\N^2$. Let $\phi:\Sp_r\SR^{k}_{[l]}(\R^2)\rightarrow R$ be a ring homomorphism where $R$ is a real closed extension of $\R$. For any $f\in\SR^{k}_{[l]}(\R^2)$, one has
$$\phi(f)=f_R\left(\phi(x_1),\phi(x_2)\right)$$
\end{prop}
\begin{proof}
The case $l=0$ is trivial since in this case $\SR^{k}_{[0]}(\R^2)$ is the ring of regular functions and any morphism $\R[x_1,x_2]\rightarrow R$ clearly admits a unique factorization through $\SR^{k}_{[0]}(\R^2)$.\par
For any $l\not=0$, the proof is quite the same as the proof of Proposition \ref{SubstitutionRegulue} although we just emphasize a point to take into account. 
Indeed, we use the  \L ojaciewicz property in $\SR^{k}(\R^2)$ since there is no \L ojaciewicz property in $\SR^{k}_{[l]}(\R^2)$ (cf section \ref{Radicalprincipality}) ! We get then an identity
$f^N=(x_1^2+x_2^2)g$ where, a priori, $g\in\SR^{k}(\R^2)$. But, it is clear from that identity that in fact $g\in\SR^{k}_{[l]}(\R^2)$ and the conclusion follows. 
\end{proof}

We derive likewise the Artin-Lang property in $\SR_{[l]}^{k}(\R^2)$:
\begin{prop}\label{ArtinLangbis}
Let $(k,l)\in\N^2$. Let $f_1,\ldots,f_r$ in $\SR_{[l]}^{k}(\R^2)$ and set $$S=\{x\in\R^n\mid f_1(x)>0,\ldots,f_r(x)>0\},$$
and 
$$\widetilde{S}=\{\alpha\in\Sp_r\SR_{[l]}^{k}(\R^2)\mid f_1(\alpha)>0,\ldots,f_r(\alpha)>0\}.$$
Then,
$\widetilde{S}=\emptyset$ if and only if $S=\emptyset$.
\end{prop} 

\subsection{Radical ideals and Nullstellensatz}

In \cite{FHMM}, they proved the weak Nullstellensatz for the ring
$\SR^{k}(\R^n)$ with $(k,n)\in(\N\cup\{\infty\})\times \N^*$ and the
strong Nullstellensatz for the ring $\SR^{k}(\R^n)$ with
$(k,n)\in\N\times \N^*$.
Since their proof of the weak Nullstellensatz (\cite[Prop. 5.23]{FHMM})
for the ring
$\SR^{k}(\R^n)$ 
uses the \L ojasiewicz property, we give here a slightly different
one which has the advantage to be also valid in the ring $\SR^{k}_{[l]}(\R^2)$ .

\begin{prop} (Weak Nullstellensatz for  $\SR^{k}(\R^n)$)\\
\label{weak1}
Let $(k,n)\in\N\times\N^*$.  Let $I\subset \SR^{k}(\R^n)$ be an
ideal. Then, $\Z(I)=\emptyset$ if and only if $I= \SR^{k}(\R^n)$.
\end{prop}

\begin{proof}  Since the $k$-regulous topology is noetherian
  \cite[Thm. 4.3]{FHMM} (a result which does not require \L ojasiewicz property), there
  exists a finite number of $f_i\in I$, $i=1,\ldots,m$, such that
  $\Z(I)=\Z(f_1)\cap\cdots\cap \Z(f_m)$. Hence, $\Z(I)=\Z(f)$ with
  $f=f_1^2+ \cdots +f_m^2\in I$. If we suppose $\Z(I)=\emptyset$ then
  $f$ is a unit and thus $I= \SR^{k}(\R^n)$. The converse implication
  is clear.
\end{proof}

The same proof gives:
\begin{prop} (Weak Nullstellensatz for $\SR_{[l]}^{k}(\R^2)$)\\
\label{weak2}
Let $(k,l)\in\N^2$.  Let $I\subset \SR_{[l]}^{k}(\R^2)$ be an
ideal. Then, $\Z(I)=\emptyset$ if and only $I= \SR_{[l]}^{k}(\R^2)$.
\end{prop}

From the weak
Nullstellensatz, one may deduce the description of the maximal ideals 
in $\SR^{k}(\R^n)$ (resp. $\SR^{k}_{[l]}(\R^2)$). For $a\in\R^n$
(resp. $a\in\R^2$), let us denote $\I (a)=\{f\in \SR^{k}(\R^n)\mid
f(a)=0\}$ (resp. $\I (a)=\{f\in \SR_{[l]}^{k}(\R^2)\mid
f(a)=0\}$). 

\begin{prop}
\label{max1} 
Let $(k,n)\in\N\times\N^*$. Let $I\subset \SR^{k}(\R^n)$ be an
ideal. Then, $I$ is maximal if and only if there exists $a\in \R^n$
such that $I=\I (a)$.
\end{prop}

\begin{proof} 
Let $a\in\R^n$, then $\I(a)$ is clearly maximal. Conversely, assume
$I$ is maximal. Since $I$ is proper then $\Z(I)\not=\emptyset$
by Proposition \ref{weak1}. Let $a\in \Z(I)$, then we have $I\subset\I(a)$ which gives $I=\I(a)$ by maximality of $I$.
\end{proof}

The same proof gives:
\begin{prop}
\label{max2} 
Let $(k,l)\in\N^2$. Let $I\subset \SR_{[l]}^{k}(\R^2)$ be an
ideal. Then, $I$ is maximal if and only if there exists $a\in \R^2$
such that $I=\I (a)$.
\end{prop}

The key tool to obtain the strong Nullstellensatz for the ring $\SR^{k}(\R^n)$, with
$(k,n)\in\N\times \N^*$, is the property that a radical ideal is real. It has been proved in
\cite[Prop. 5.7]{FHMM}. 

\begin{lem} \label{radreal}
Let $(k,n)\in\N\times\N^*$. Let $I$ be a radical
ideal in $\SR^{k}(\R^n)$. Then $I$ is a real ideal. 
\end{lem}

One proof of this fact can be deduced immediately from the \L
ojaciewicz property. 
Another proof, from \cite[Prop. 5.7]{FHMM} is based on the fact
that we can compose $k$-regulous maps \cite[Cor. 4.14]{FHMM} which is
a consequence of \cite[Thm. 4.1]{FHMM}. 
For a given $f\in\SR^{k}(\R^n)$, Theorem 4.1 in \cite{FHMM} ensures
the existence of finite stratification of $\R^n$ in Zariski locally
closed subsets $S_i\subset \R^n$ 
such that the restriction of $f$ to each $S_i$ is regular. The theorem
can be proved using Hironaka's theory of resolution of singularities
but do not require the use of \L ojasiewicz property. \par

One may now give a proof of the Nullstellensatz which does not use the
\L ojasiewicz property :

\begin{thm} (Nullstellensatz for $\SR^{k}(\R^n)$)\\
\label{nullstellensatz}
Let $(k,n)\in\N\times\N^*$ and $I$ be an ideal in $\SR^{k}(\R^n)$. Then, $$\Rad{I}=\I(\Z(I)).$$
\end{thm}

\begin{proof}
The inclusion $\Rad{I}\subset\I(\Z(I))$ is trivial.

By noetherianity of the $k$-regulous topology, we get the existence of
$f\in I$ such that $\Z(f)=\Z(I)$. Let $g\in \I(\Z(I))$. We have
$\Z(f)\subset \Z(g)$. By hypothesis
$$\{x\in\R^n|\,\,f(x)=0,\,g(x)\not=0\}=\emptyset.$$ Then by Proposition \ref{ArtinLang}, 
$$\{\alpha\in\Sp_r \SR^{k}(\R^n)|\,\,f(\alpha)=0,\,
g(\alpha)\not=0\}=\emptyset$$
By the formal Positivstellensatz \cite[Prop. 4.4.1]{BCR}, there
exist $h_1,\ldots,h_m,h\in \SR^{k}(\R^n)$ and $e\in\N^*$ such that 
$$\sum_{i=1}^m h_i^2+g^{2e}+fh=0.$$ Hence $g$
belongs to the real radical of the ideal $(f)$. By Lemma \ref{radreal}, $\RadRe{(f)}=\Rad{(f)}\subset\Rad{I}$.
\end{proof}

We complete the understanding of radical ideals in continuation of Lemma \ref{radreal}
by proving that a non-trivial principal ideal of
$\SR^{k}(\R^n)$ is never radical.

\begin{prop}
\label{princnotrad}
Let $(k,n)\in\N\times\N^*$ and $f\in \SR^{k}(\R^n)$ such that $(f)$ is neither the zero ideal nor the whole ring. Then, the principal ideal generated by $f$ is not radical.
\end{prop}

\begin{proof}
Suppose $(f)$ is radical.
By hypothesis, $\Z(f)\not=\emptyset$.

First assume $\dim (\Z(f))\geq 1$. We may suppose that the origin $o$
of $\R^n$ is not an isolated point in
$\Z(f)$. Let $q=x_1^2+\cdots +x_n^2$. By \L ojasiewicz property \cite[Lem. 5.1]{FHMM}, there
exists a $N\geq 1$ such that $\frac{f^N}{q}\in \SR^{k}(\R^n)$. We
can choose $N$ minimal for the previous property. Since $o$ is not
isolated in $\Z(f)$, we have $\frac{f^N}{q}(o)=0$. Thus,
$\frac{f^N}{q}\in\I (\Z(f))$. By Theorem \ref{nullstellensatz} and since the ideal $(f)$ is radical by assumption, we have $\I
(\Z(f))=(f)$ and we get an identity $f^N=fqh$ with $h\in
\SR^{k}(\R^n)$ i.e. $h=\frac{f^{N-1}}{q}\in \SR^{k}(\R^n)$. This is
imposssible by our choice of $N$.

Assume now $\dim (\Z(f))=0$ i.e. $\Z(f)=\cup_{i=1}^m {a_i}$ is a
disjoint union of points $a_i=(a_{i,1},\ldots,
a_{i,n})\in\R^n$.
Let $q=\prod_{i=1}^m (\sum_{j=1}^n
(x_j-a_{i,j})^2)\in\R[x_1,\ldots,x_n]$. Since $q\in \I(\Z(f))$,
using Theorem \ref{nullstellensatz}  and the hypothesis $\Rad{(f)}=(f)$, we get an identity  
$q=fh$ with $h\in \SR^{k}(\R^n)$. For $i=1,\ldots,m$, we have
$\ord_{a_i}(q)=2$, $\ord_{a_i}(f)$ is even and thus
$h(a_i)\not= 0$ and $\ord_{a_i}(f)=2$. Since $\Z(h)\subset
\Z(q)=\Z(f)$ then $h$ is a unit in $\SR^{k}(\R^n)$. Hence we may
assume $f=q$. 
For $i=1,\ldots,m$, let $H_i$ be an hyperplane given by the linear
equation $h_i=0$ such that $a_i\in H_i$ and $a_j\not\in H_i$
for $i\not= j$. Since $h_1\cdots h_m\in \I(\Z(q))$, using again
Theorem \ref{nullstellensatz}, we get $\frac{h_1\cdots h_m}{q}\in
\SR^{k}(\R^n)$. This is impossible since $\ord_{a_i}(h_1\cdots h_m)=1$ and thus $\frac{h_1\cdots
  h_m}{q}$ is not continuous at $a_i$.
\end{proof}

On the contrary to Lemma \ref{radreal}, in the ring $\SR^{k}_{[l]}(\R^2)$ we do not have in general $\Rad{I}=\RadRe{I}$. We give below some counterexamples which also disprove the Nullstellensatz in $\SR^{k}_{[l]}(\R^2)$. 

\begin{ex}
For $l=0$, the ring $\SR^{k}_{[l]}(\R^2)$ is the ring $\SR^{\infty}(\R^2)$ of regular functions on
$\R^2$. The ideal $I=(x^2+y^2)$ is prime and hence radical $I=\Rad I$. And one can cheek that $x\in \I(\Z(I))$ but $x\not\in
I$. Namely $\Rad{I}\not=\I(\Z(I))$.
\par
\par
For $l=1$, let us consider $I=(x^2+y^4)$ in $\SR^{k}_{[1]}(\R^2)$. We have
$x\in \I(\Z(I))$ but $x\not\in \Rad I$. Indeed, suppose
$x\in\Rad I$, then there exists $N\in\N$ such that $x^N\in
(x^2+y^4)$ i.e. the continuous extension of the rational function
$\frac{x^N}{x^2+y^4}$ belongs to $\SR^{k}_{[1]}(\R^2)$, a
contradiction. Thus we get $\Rad{I}\not=\I(\Z(I))$.
Let $J=\Rad{I}$. Since
$x\not\in J$ (proved above), the radical ideal $J$ of
$\SR^{k}_{[1]}(\R^2)$
is not real.\par
For $l>1$, it suffices to consider the ideal $I=(x^2+y^{2(l+1)})\quad\square$.
\end{ex}

Nevertheless, one may state a real Nullstellensatz.

\begin{thm} (Real Nullstellensatz for $\SR^{k}_{[l]}(\R^2)$)\\
\label{realnullstellensatz}
Let $(k,l)\in\N^2$ and $I$ be an
ideal in $\SR^{k}_{[l]}(\R^2)$. Then,
$$\RadRe{I}=\I(\Z(I)).$$
In particular $\Rad{I}=\I(\Z(I))$ if $I$ is a real ideal.
\end{thm}
\begin{proof}
Let $f\in \I(\Z(I))$. By hypothesis we have $$\{x\in\R^2\mid \forall g\in I,g(x)=0\;{\rm and}\;f(x)\not=0\}=\emptyset.$$
By Proposition \ref{ArtinLangbis}, we get $$\{\alpha\in \Sp_r\SR^{k}_{[l]}(\R^2)\mid \forall g\in I,g(\alpha)=0\;{\rm and}\;f(\alpha)\not=0\}=\emptyset.$$
Then, using the formal Positivstellensatz \cite[Prop. 4.4.1]{BCR}, one gets an identity
$$f^{2m}+\sum_{i=1}^rh_i^2+g=0$$
where the $h_i$'s are in  $\SR^{k}_{[l]}(\R^2)$ and $g$ in $I$.
This exactly says that $f\in\RadRe{I}$ and hence shows the non-obvious inclusion 
$\I(\Z(I))\subset\RadRe{I}.$ 
\end{proof}



As a consequence of the formal Positivstellensatz, one deduces more generally :
\begin{thm}\label{PSSRegulous}
Let $(k,n)\in\N\times\N^*$.
Let $f,g_1,\ldots,g_r,h_1,\ldots,h_s$ in $\SR^{k}(\R^n)$ such that $f\geq 0$ on the set $\{x\in\R^n\mid g_1(x)=0,\ldots,g_r(x)=0,h_1(x)\geq 0,\ldots,h_s(x)\geq 0\}$.
Then, there exists an identity 
$$c_1f=c_2+d$$
where $d$ lies in the ideal generated by the $g_i$'s and $c_1,c_2$ are
in the cone generated by the $h_i$'s. Namely $c_1$ and $c_2$ are sums
of elements of the form $sh_1^{\epsilon_1}\ldots h_s^{\epsilon_s}$ where
$s$ is a sum of squares in $\SR^{k}(\R^n)$ and the $\epsilon_i$'s are in $\{0,1\}$.
\end{thm} 
\begin{proof} 
By assumption,
$$\{x\in\R^n\mid f(x)\not=0,g_1(x)=0,\ldots,g_r(x)=0,-f(x)\geq 0,h_1(x)\geq 0,\ldots,h_s(x)\geq 0\}=\emptyset.$$

Hence, by Proposition \ref{ArtinLang}
$$\begin{array}{rcl}
\{\alpha\in\Sp_r\SR^{k}(\R^n)&\mid& f(\alpha)\not=0,g_1(\alpha)=0,\ldots,g_r(\alpha)=0,\\
&&\\
&&-f(\alpha)\geq 0,h_1(\alpha)\geq 0,\ldots,h_s(\alpha)\geq 0\}=\emptyset.
\end{array}$$

Then, by the formal Positivstellensatz \cite[Prop. 4.4.1]{BCR}, we have an identity
$$f^{2m}+u-fv+w=0$$
where $m$ is an integer, $u$ is in the ideal generated by the
$g_i$'s and $v,w$ are in the cone generated by the $h_i$'s in $\SR^{k}(\R^n)$. Hence, we get
$$f(f^{2m}+w)=f^2v-fu$$
and we are done with $c_1=f^{2m}+w$, $c_2=f^2v$ and
$d=-fu$.
\end{proof}

\begin{rem} \label{PSSRegulousbis}The statement of the previous theorem remains valid with the ring
  $\SR^{k}_{[l]}(\R^2)$ instead of $\SR^{k}(\R^n)$.
\end{rem}

\subsection{Radical principality and \L ojasiewicz property in $\SR^k(\R^n)$ and $\SR^k_{[1]}(\R^2)$}\label{Radicalprincipality}

We recall that a commutative ring $A$ is said to be radically principal if for any
ideal $I\subset A$, there exists $f\in I$ such that $\Rad{(f)}=\Rad{I}$.

As previously mentioned, in \cite{FHMM} a direct application of \L
ojasiewicz inequality shows a \L ojasiewicz property that can be formulated as:

\begin{thm} \cite[Thm. 5.21]{FHMM}
\label{radprinc}
Let $(k,n)\in\N\times\N^*$. The ring
$\SR^{k}(\R^n)$ is radically principal.
\end{thm}

\begin{rem} By the previous theorem and Proposition
  \ref{princnotrad}, we know that a non-trivial radical ideal of $\SR^{k}(\R^n)$
  is never principal but is still the radical of a principal
  ideal.
\end{rem}

In fact, looking at the proof of \cite[Thm. 5.21]{FHMM}, one has even more :

\begin{rem}\label{radprinc+}
If $I$ is an ideal of $\SR^{k}(\R^n)$, then for any $f\in I$ such that $\Z(f)=\Z(I)$ and any $g$ in $I$, there exist an integer $N$ and $h\in\SR^{k}(\R^n)$ such that $g^N=fh$. 
\end{rem}

Beware that this last Remark is clearly false in
$\SR_{[1]}^0(\R^2)$. For instance take $I=(x, x^2+y^4)$, $g=x$,
$f=x^2+y^4$. Nevertheless, we shall prove a weaker version of this \L
ojaciewicz property replacing the universal quantifier by an
existential one. Namely, we will prove that there exists a convenient
$f\in I$ satisfying $\Z(f)=\Z(I)$ such that, for any $g\in I$, a \L
ojasiewicz property exists.

\begin{ex} Let $I$ be the ideal generated by $x$ and $y$ in
  $\SR_{[1]}^0(\R^2)$. Let $f=x^2+y^2$, then $f\in I$ and
  $\Z(f)=\Z(I)$. Now take $g\in I$. By \L ojasiewicz property there
  exist an integer $N$ and $h\in\SR^{0}(\R^2)$ such that $g^N=fh$. By
  Theorem \ref{equivfctregapres1eclat}, $h=\frac{g^N}{x^2+y^2}\in
  \SR_{[1]}^0(\R^2)$.
\end{ex}
  
Let us start with some settings.
\begin{defn}
For an ideal $I$ of $\SR^k_{[1]}(\R^2)$, we denote by $N_I\subset \R[x,y]$ the ideal of numerators of $I$, namely
$$N_I=\left\langle p\in\R[x,y]\mid\exists q\in\R[x,y]\setminus\{0\},\;\frac{p}{q}\in I \right\rangle.$$
\end{defn}

Note that if $N_I=(p_1,\ldots,p_n)$ in $\R[x,y]$, then the great commun divisor $d_I$ of the ideal $N_I$ exists and satisfies $d_I=\gcd(p_1,\ldots,p_n)=\gcd(N_I)$, since $\R[x,y]$ is a Unique Factorization Domain. Moreover $N_I\subset (d_I)$ but the reverse inclusion is not necessarily satisfied since $\R[x,y]$ is not a principal domain.\par 
Note also that, viewed in $\SR^k_{[l]}(\R^2)$, we clearly have the inclusion $N_I\subset I$. In the following, we often write an element $f\in I$ in a standard form as  
$f=\frac{p}{q}=\frac{d_I a}{q}$ where $p,q,a$ are in $\R[x,y]$ and $p$ and $q$ are coprime.

\air
We recall from Definition \ref{definite} that a polynomial $p$ is said
to be locally positive definite at a point $a\in\R^n$ if the
homogeneous component of lowest degree of $p$ at the point $a$ is
positive definite.

For the following we will need 

\begin{prop}\label{ZerosIdealNumerators}
Let $I$ be an ideal of $\SR_{[1]}^k(\R^2)$ and $d=d_I$ be the gcd of the associated ideal of numerators $N_I$ in $\R[x,y]$.
Then, 
 there is an $f$ in $I$ such that $\Z(f)=\Z(I)$ and $f=d^2\frac{p}{q}$
 where $p,q$ are polynomials in $\R[x,y]$ and $p$ is locally positive
 definite everywhere and $\Z(p)$ is a finite set.
\end{prop}

\begin{proof}
First of all, by the noetherianity of the regulous topology (\cite[Thm. 4.3]{FHMM}), we know that there is a function $f_I\in\SR_{[1]}^k(\R^2)$ such that $\Z(I)=\Z(f_I)$.
Let us write $f_I=d\frac{r}{s}$ where $r$ and $s$ are coprime.\par 

If $N_I=(p_1,\ldots,p_n)$ in $\R[x,y]$, then let us set $p_i=d\overline{p_i}$. Hence we have $\gcd(\overline{p_1},\ldots,\overline{p_n})=1.$
In the principal ring $\R(x)[y]$, we get a B\'ezout type indentity that gives rise to another identity in $\R[x,y]$

$$\overline{p_x}=u_1\overline{p_1}+\ldots+u_n\overline{p_n}$$
where $\overline{p_x}\in\R[x]\setminus\{0\}$, $u_1,\ldots, u_n\in \R[x,y]$.

Likewise, working in $\R(y)[x]$, we get an identity
$$\overline{p_y}=v_1\overline{p_1}+\ldots+v_n\overline{p_n}$$
where
$\overline{p_y}\in\R[y]\setminus\{0\}$, $v_1,\ldots, v_n\in \R[x,y]$.

Up to multiplying our identities by linear factors of type $x-\lambda$ (respectively $y-\mu$), one may assume that $\overline{p_x}$ (respectively $\overline{p_y}$) have same multiplicity $\nu$ at each real zero.


Set $f_I=d \overline{f_I}=d\frac{r}{s}$, $p_x=d\overline{p_x}$ and $p_y=d\overline{p_y}$ and define a function $f$ by
$$f=f_I^2+p_x^2+p_y^2=d^2(\overline{f_I}^2+\overline{p_x}^2+\overline{p_y}^2)=d^2\frac{p}{s^2},$$
where $$p=r^2+s^2(\overline{p_x}^2+\overline{p_y}^2).$$

Up to changing $f_I$ with $f_I^{\nu+1}$, one may assume that at any point $a$ of $\Z(I)$, one has $\ord_a(\overline{f_I})>\nu$.

\vskip 0.2cm Let us check now that $f$ is convenient. Note that
clearly $f\in \SR_{[1]}^k(\R^2)$.
One may readily check that 
$\Z(f)=\Z(I)$ since $\Z(I)\subset \Z(N_I)\subset \Z(p_x)\cap \Z(p_y)$.\par
One has now to check that the polynomial $p$ is locally positive
definite everywhere. Since $p$ is clearly non-negative on $\R^2$, it
is sufficient to prove that $p$ is locally definite on $\Z(p)$.\par
First, let us consider a point $a_0=(x_0,y_0)$ in $\mathcal Z(p)\cap \mathcal Z(I)$.
Let $s_{a_0}$ be the homogeneous component of lowest degree of $s$ at $a_0$.
By assumption $\ord_{a_0} (\overline{f_I})>\nu $ hence
$\ord_{a_0}(r)>\ord_{a_0}(s)+\nu$, 
$\ord_{a_0}(r^2)>\ord_{a_0}(s^2(\overline{p_x}^2+\overline{p_y}^2))$ and hence one has
$$p=r^2+s^2(\overline{p_x}^2+\overline{p_y}^2)\sim_{a_0} s_{a_0}^2\left((\overline{p_x})_{a_0}^2+(\overline{p_y})_{a_0}^2\right)=s_{a_0}^2\left((\alpha(x-x_0))^{2\nu}+(\beta(y-y_0))^{2\nu}\right),$$ 
for some $\alpha,\beta\in\R$. 

Since $s$ is definite at $a_0$, it proves that $p$ is locally positive definite at $\Z(p)\cap \Z(I)$. \par

Let us consider now $a_0\in \Z(p)\setminus \Z(I)$. We have $a_0\in
\Z(s)$ since $f(a_0)\not=0$. At the neighbourhood of $a_0$, since
$f\in \SR_{[1]}^k(\R^2)$, 
one has $s\sim_{a_0}s_{a_0}$ by Lemma \ref{lem-polaire} and also
$d_{a_0}^2p_{a_0}=f({a_0}) s^2_{a_0}$ where $p_{a_0}$
(resp. $d_{a_0}$) denotes the homogeneous component of lowest degree
of $p$ (resp. $d$) at $a_0$. Hence $p_{a_0}$ is locally
positive definite at ${a_0}$ by Lemma
\ref{productdefinitepolynomials}.

We have shown that $p$ is locally positive definite everywhere. 
To conclude the proof, let us note that since
$p=r^2+s^2(\overline{p_x}^2+\overline{p_y}^2)$, then $\Z(p)\subset
\Z(s)\cup \left(\Z(\overline{p_x})\cap \Z(\overline{p_y})\right)$ and
hence $\Z(p)$ is a finite set. 
\end{proof}

Now we may state the \L ojasiewicz property :
\begin{thm}\label{Lojaregblowdim2}
Let $I$ be an ideal of $\SR_{[1]}^k(\R^2)$. Then, we can find $f\in I$ such that for any $g\in I$ there exist  $N\in\N$ and $h\in\SR_{[1]}^k(\R^2)$ such that $g^N=f h$.
\end{thm}
\begin{proof}
First, let us consider $f$ as in Proposition \ref{ZerosIdealNumerators} : 

$$f=\frac{d^2 p}{q}$$
and
$$g=\frac{d r}{s}$$ 
where $r$ and $s$ are coprime in $\R[x,y]$.

Using the \L ojasiewicz property in the $k$-regulous setting, we have 
$g^N=f h$ where $h$ is regulous in $\SR^k(\R^2)$.

We have
$$h=\frac{g^N}{f}=\frac{d^{N-2}q}{p}\left(\frac{r}{s}\right)^N.$$
It follows from Proposition \ref{ZerosIdealNumerators} and the fact
that $g\in \SR_{[1]}^k(\R^2)$ that the denominator $ps^N$ of $h$ is
locally definite on $\R^2$.
By Theorem \ref{equivfctregapres1eclat}, we are done.
\end{proof}

\begin{rem}
The element $f$ cannot necessarily be choosen as a polynomial since the inclusion $\Z(I)\subset \Z(N_I)$ may be strict. Take for example $I=\left(1-\frac{x^3}{x^2+y^2}\right)\subset\SR^0_{[1]}(\R^2)$.
\end{rem}

To end this section, let us reformulate Theorem \ref{Lojaregblowdim2} in term of radical principality of the ring $\SR^k_{[1]}(\R^2)$ :

\begin{thm}\label{radicprincipdim2}
Any radical ideal of $\SR^k_{[1]}(\R^2)$ can be written as the radical of a principal ideal. Namely, for any 
ideal $I$ of $\SR_{[1]}(\R^2)$, there exists an element  
$f\in I$ such that $\sqrt{I}=\sqrt{(f)}$.
\end{thm}

\section{Hilbert's $17$-th problem and Positivstellens\"atze in rings
  of regulous functions}
\label{Hilbert17}
\subsection{Algebraic certificates of positivity}

The interest into sums of squares identities in real algebra goes back to Hilbert and his famous list of problems. Namely the Hilbert's 17-th problem ask whenever a non-negative polynomial on $\R^n$ is a sum of squares of rational functions. It has been answered by the affirmative by Artin in 1927 : \par

\vskip 0,2cm
\noindent{\bf Theorem}\;
{\it If $f$ is polynomial in $\R[x_1,\ldots,x_n]$ which is non-negative on $\R^n$, then $f$ is a sum of squares of rational functions in $\R(x_1,\ldots,x_n)$.}
\vskip 0,2cm

In general, $f$ is not a sum of squares in the ring $\R[x_1,\ldots,x_n]$, for example the Motzkin polynomial $z^6+x^4y^2+x^2y^4-3x^2y^2z^2$ is non-negative but not a sum of squares of polynomials.\par

Let us mention that there are several Positivstellens\"atze in some other geometric settings:
\begin{itemize}
\item Any non-negative Nash function on an affine Nash manifold is a sum of squares of quotients of Nash functions (\cite[8.5.6]{BCR}).
\item Any non-negative analytic function on a compact affine analytic manifold is a sum of squares of meromorphic functions (\cite{Ru} and also \cite{Ri} for a local version).
\end{itemize}

There exist also some Positivstellens\"atze {\it without denominator}, namely when the elements of the sum of squares lie already in the ring  and not just in the associated field of fractions. These cases are quite rare, let us state some of them : 

\begin{itemize}
\item Any non-negative analytic function on an affine analytic surface is a sum of squares of analytic functions (\cite{BKS}). 
\item Any non-negative $C^{2k}$ function defined in an interval of $\R$ is the sum of the squares of two $C^k$ functions (\cite{Bo}).
\end{itemize}

Hence, it is worth studying which Positivstellensatz holds in the rings $\SR^k(\R^n)$ and $\SR^k_{[l]}(\R^2)$.
In the sequel, we give a Positivstellensatz without denominator in $\SR^0(\R^n)$ which is obtained quite straightforward using standard arguments, and next we show, by some new arguments, that there exists also a Positivstellensatz without denominator in 
$\SR^0_{[1]}(\R^2)$.


\subsection{In the rings $\SR^0(\R^n)$ and $\SR^\infty(\R^2)$}
As a motivation, in this section we give two Positivstellens\"atze without denominators for the rings $\SR^0(\R^n)$ and $\SR^\infty(\R^2)$ which one easily derives using well known results.\par
First of all, it has already been noticed by Kreisel (\cite{Kr}) that any non-negative polynomial is a sum of squares of regulous functions. It is not difficult to generalize it a bit for regulous functions, for the convenience of the reader we produce a standard proof for this fact.

\begin{thm}\label{sosreg0}
Let $f\in\SR^0(\R^n)$ be non-negative on $\RR^n$. Then, there are some $f_i$'s in $\SR^0(\R^n)$ such that $f=\sum_{i=1}^m f_i^2$.
\end{thm}

\begin{proof}
For such an $f$ let us consider the set $S=\{x\in\R^n\mid f<0\}$ which is defined by a semialgebraic condition by Proposition \ref{kregsa=sa}.
By assumption, $S=\emptyset$. \par
Then, by Proposition \ref{ArtinLang}, 
$$\{\alpha\in \Sp_r\SR^0(\R^n)\mid f(\alpha)<0\}=\emptyset.$$
One then may conclude with the formal Positivstellensatz \cite[Prop. 4.4.1]{BCR}, to the existence of an identity
$f^{2k}-fs+t=0$ where $s$ and $t$ are sums of squares in $\SR^0(\R^n)$.\par
Hence, $f\times (f^{2k}+t)=f^2\times s$ and one has an identity 
$$f=\dfrac{f^2s\times (f^{2k}+t)}{(f^{2k}+t)^2}=\sum_i\left(\dfrac{f_i}{g}\right)^2$$ where $f_i,g$ are in $\SR^0(\R^n)$ and such that $\Z(g)\subset \Z(f)$. Then, the $\frac{f_i}{g}$'s are rational functions, continuous at any zero $a$ in $\Z(g)$ since $f$ vanishes at $a$ and hence any $f_i/g$ necessarily tends to $0$ at $a$.
\end{proof}

Now, when $k=\infty$, one has also a Positivstellensatz, at least in dimension two. For the convenience of the reader, let us recall what result leads to such a Positivstellensatz.
Since $\SR^\infty(\R^2)$ is the ring of regular functions, it appears as a consequence of the properties of so-called ``bad points'' introduced by Delzell in his thesis. \par
For a non-negative polynomial $p$ in $\R[x_1,\ldots,x_n]$, let us define the set of all the possible denominators in sums of squares identities associated to $p$ :
$${\rm HD}(p)=\{q\in\R[x_1,\ldots,x_n]\mid q^2p\in\sum\R[x_1,\ldots,x_n]^2\}.$$
It is known (\cite{De}) that there exists a polynomial $q\in {\rm HD}(p)$ such that $\Z(q)=\mathcal Z({\rm HD}(p))$ and, on the other hand, $\Z({\rm HD}(p))$ (the set of ``bad points'') has codimension at least $3$. Therefore, in dimension $2$, one has

\begin{thm} Delzell \label{nobadpointdim2}{\cite{De}} \par
Let $p$ be a non-negative polynomial in $\R[x,y]$. Then, there is an algebraic identity $q^2p=\sum p_i^2$ where 
$q,p_i\in\R[x,y]$ and $q>0$ on $\R^2$.
\end{thm}

\begin{rem}\label{unif_denom}
After Scheiderer \cite{Sc}, one may even take for $q$ a polynomial of
the form $q=(x^2+y^2+1)^N$ for some integer $N\gg 0$.
\end{rem}
As a consequence, one gets a Positivstellensatz without denominator in the ring
$\SR^\infty(\R^2)$ :

\begin{thm}\label{reguluesosdim2}
Let $f\in\SR^\infty(\R^2)$ be non-negative on $\R^2$. Then, there are some $f_1,\ldots,f_m$ in $\SR^\infty(\R^2)$ such that $f=\sum_{i=1}^m f_i^2$.
\end{thm}

\begin{proof}
Let us write $f=g/h=gh/h^2$ where $g,h\in\R[x,y]$ and $h$ does not vanish. It suffices then to apply the previous Theorem \ref{nobadpointdim2} to the polynomial $gh$.  
\end{proof}

It is then natural to look for some Positivstellens\"atze without denominator in the rings $\SR^k(\R^n)$ and $\SR^k_{[l]}(\R^2)$. In the following section, we answer to the case $\SR^0_{[1]}(\R^2)$.\par 
Unfortunately, it is not clear how to generalize our construction to $\SR_{[l]}^k(\R^2)$ since for $l>1$ one should have an algebraic characterization as in Theorem \ref{equivfctregapres1eclat}.
The generalization to the case $k>0$ seems also appealing but it is not clear how to follow the conditions of $k$-regularity through the sum of squares identities of our construction. 

\subsection{Hilbert's $17$-th problem in the ring $\SR^0_{[1]}(\R^2)$}

The main natural question that we address here is to find a Positivstellensatz without denominator in the ring $\SR^0_{[1]}(\R^2)$, namely the fact that any rational function $f$ in $\SR^0_{[1]}(\R^2)$ which is non-negative everywhere can be written as a sum of squares in $\SR^0_{[1]}(\R^2)$.\par

First note that a formal argument as in the ring of regulous functions does not work any longer because it does not provide any control on the denominator.\par
Indeed, we would get, using the formal Positivstellensatz (Remark \ref{PSSRegulousbis}), an identity of the form
$f^{2k}-fs+t=0$ where $s$ et $t$ are sums of squares in $\SR_{[1]}^0(\R^2)$. Hence, we would get an identity
$f=\sum f_i^2$ where the $f_i$'s have the form $\frac{g_i}{f^{2k}+t}$
with $g_i$ in $\SR_{[1]}^0(\R^2)$, but, this identity does not show
whether $f_i$ is or not in $\SR_{[1]}^0(\R^2)$ (it only gives that
$f_i\in \SR^0(\R^2)$).

\subsubsection{Under a flatness hypothesis}

When $f=p/q\in\SR^0(\R^2) $, we have seen in Subsection 1.3 that the condition $\Z(q)\subset \Z(f)$ implies that, for any integer $k$, there is a power of $f$ which is $k$-flat and of class $C^k$ everywhere. It appears that in this case it is easy 
to obtain a solution to the $17$-th Hilbert problem.
\begin{prop}\label{plateapresuneclat}
Let $f=p/q$ be a non-negative function in $\SR^0_{[1]}(\R^2)$ such that $\Z(q)\subset \Z(f)$. Then, $f=\sum f_{i}^2$ where the $f_{i}$'s are regulous functions in $\SR^0_{[1]}(\R^2)$.
\end{prop}
\begin{proof}
First of all, since $\Z(q)$ is finite, we know that $q$ keeps the same sign on the whole $\R^2$, hence we may assume that $q$ (and also $p=qf$) are both non-negative polynomials.\par
By  \ref{nobadpointdim2}, there exist non-vanishing polynomials $a$ and $b$ in $\R[x,y]$ such that $a^2p=\sum p_i^2$, $b^2q=\sum q_j^2$. Hence $f=\sum f_{i,j}^2$ where $$f_{i,j}=\frac{b p_i q_j}{a\sum q_k^2}.$$
The $f_{i,j}$'s are rational functions which are continuous (they tend to $0$) at each $x$ in $\Z(\sum q_k^2)$ since $\Z(q)\subset \Z(f)$.
Moreover, their denominator is essentially the denominator of $q$, hence $f_{i,j}$'s are in $\SR^0_{[1]}(\R^2)$.\par
\end{proof}

\begin{rem}
\begin{enumerate}
\item[(i)]
Under this flatness hypothesis $\mathcal Z(q)\subset \mathcal Z(f)$, let us notice that
one may ``choose'' the denominator in the sum of squares
identity. Namely, by Remark \ref{unif_denom}, one may even take
$a=b=(x^2+y^2+1)^N$ for an $N\gg 0$.
\item[(i)]
The condition $\Z(q)\subset \Z(f)$ is equivalent to $\Z(p)=\Z(f)$, which means in particular that the zero set of $f$ is an algebraic set.
\end{enumerate}
\end{rem}

An interesting corollary :
\begin{prop}
If $f$ is a non-negative function in $\SR_{[1]}^0(\R^2)$ with
$\mathcal Z(q)\subset \mathcal Z(f)$, then $f$ is of class $C^1$ on $\R^2$.
\end{prop}
\begin{proof}
From the identity $f=\sum f_i^2$ coming from \ref{plateapresuneclat}, one deduces that $\partial_vf=2\sum f_i\cdot(\partial_vf_i)$. We know that $\partial_vf_i$
are locally bounded by Proposition \ref{locbound}, hence
$\partial_vf\to 0$ at any point of $\pol(f)$. The proof follows now
from Theorem \ref{thm-top}.
\end{proof}

Our goal in the sequel is to get rid off the condition $\Z(q)\subset
\Z(f)$ in \ref{plateapresuneclat}.

\subsubsection{General case}
In the general case, one may construct successive algebraic identities
where we get rid off the undesirable poles one by one by an induction
process. First, some notation.
Let $f_1$ and $f_2$ be some regulous functions, we set
$$\Zd_{f_1}(f_2)=\{a\in\R^2\mid f_2(a)=0,f_1(a)\not=0\}.$$
Let $f\in \SR_{[1]}^0(\R^2)$ be non-negative on $\R^2$. By Theorem
\ref{nobadpointdim2}, we may write $f=\dfrac{p}{q}$ with $p$ and $q$
some sums of squares of polynomials in two variables.\par
If $\Z_f(q)=\emptyset$ i.e. if $\Z(q)\subset \Z(f)$ then we are able
to write $f$ as a sum of squares of elements in $\SR_{[1]}^0(\R^2)$
(Proposition \ref{plateapresuneclat}).\par
If $\Z_f(q)\not=\emptyset$ then choose $a\in \Z_f(q)$. The idea, in
the following, is to construct a new representation $f=\dfrac{g}{h}$
with $g$ and $h$ some sums of squares of elements in
$\SR_{[1]}^0(\R^2)$ such that $$\Z_f(h)=\Z_f(q)\setminus\{a\}.$$
Repeating this process we will get a representation of $f$ as a
quotient of elements in $\SR_{[1]}^0(\R^2)$ such that $\Z_f$ of the
denominator is empty. The key tool to obtain such representation of
$f$ is the next lemma.\par

Before stating the lemma, let us emphasis a key algebraic identity 
\begin{equation}\label{identitepoly}
\sum_{i=1}^n X_i^2\cdot\sum_{i=1}^nY_i^2=\left(\sum_{i=1}^n X_i Y_i\right)^2+\sum_{1\leq i<j\leq n}(X_iY_j-X_jY_i)^2.
\end{equation}

\begin{lem}\label{keyinductionsos}
Let $f$ be a regulous function in $\SR_{[1]}^0(\R^2)$. Assume that there exist some integers $k,m$, some polynomials
$p_i$, $q_j$, an identity $f=\frac{p}{q}$ and a non-negative
polynomial $d$ which is locally positive definite everywhere such that
\begin{itemize}
\item $p=\sum_{i=1}^k p_i^2$ and $q=\sum_{j=1}^m q_j^2$,
\item $q$ is locally positive definite everywhere,
\item $\frac{p_i}{d}$ and $\frac{q_j}{d}$ are continuous rational
  functions (they are in $\SR_{[1]}^0(\R^2)$ and hence $\frac{q}{d^2}$
  is also in $\SR_{[1]}^0(\R^2)$).
\end{itemize}
Assume there exists $a$ in $\R^2$ such that $q(a)=0$, $f(a)\not=0$, $d(a)\not=0$ and
set $e=d\times q_{a}$ where $q_{a}$ is the homogeneous component of
smallest degree of $q$ at the point $a$. It means in particular that $a\in \Zd_f\left(\frac{q}{d^2}\right)$. \par
Then, there are some integers $l,n$, some polynomials $r_i$, $s_j$ and an identity 
$f=\frac{r}{s}$ such that
\begin{itemize} 
\item $r=\sum_{i=1}^{l} r_i^2=p\times q_a$ and $s=\sum_{j=1}^{n} s_j^2=q\times q_a$,
\item $s$ is locally positive definite everywhere,
\item $\dfrac{r_i}{e}$ and $\dfrac{s_j}{e}$ are continuous rational functions
(and hence  $\frac{s}{e^2}$ and $\frac{r}{e^2}$ are also in $\SR_{[1]}^0(\R^2)$),
\item $\Zd_f\left(\frac{s}{e^2}\right)= \Zd_f\left(\frac{q}{d^2}\right)\setminus \{a\}$. 
\end{itemize}
\end{lem}

\begin{proof}
Let us assume for simplicity that $d(a)=1$ and $f(a)=1$.\par
Looking at the homogenous components of $q_i$ at $a$ one may write 
$q_i=t_{i}+w_{i}$
where the $t_{i}$'s and the $w_{i}$'s are polynomials in $\R[x,y]$
such that, the $t_{i}$'s are homogeneous at $a$,  $\ord_{a}(t_{i})<\ord_{a}(w_{i})$ if 
$2\ord_{a}(q_{i})=\ord_a q$ and otherwise $t_i=0$ if $2\ord_{a}(q_{i})>\ord_a q$.\par
Let us set $u_i=t_i\times d$ and define $v_{i}$ such that $q_i=u_{i}+v_{i}$. Since $d(a)=1$ we still have
$\ord_{a}(u_{i})<\ord_{a}(v_{i})$ if $2\ord_{a}(q_{i})=\ord_a q$ and $u_{i}=0$ otherwise. 
Let us write
$q_a=\sum_{i=1}^{m} u_{i}^2$.

By using identity (\ref{identitepoly}), one gets
 $$\frac{q\times q_{a}}{q_{a}^2}=\frac{\left(\sum_{i=1}^{m}(u_{i}+v_{i})^2\right)\left(\sum_{i=1}^{m} u_{i}^2\right)}{\left(\sum_{i=1}^{m} u_{i}^2\right)^2}$$
 $$=\frac{\left(\sum_{i=1}^{m}u_{i}(u_{i}+v_{i})\right)^2+\sum_{1\leq i<j\leq m} \left((u_{i}+v_{i})u_{j}-u_{i}(u_{j}+v_{j})\right)^2}{\left(\sum_{i=1}^{m} u_{i}^2\right)^2}$$
 $$=\left(1+\frac{\sum_{i=1}^{m}u_{i}v_{i}}{\sum_{i=1}^{m} u_{i}^2}\right)^2+\sum_{1\leq i<j\leq m} \left(\frac{v_{i}u_{j}-u_{i}v_{j}}{\sum_{i=1}^{m} u_{i}^2}\right)^2$$
Set $s_1=\sum_{i=1}^{m} u_{i}^2+\sum_{i=1}^{m}u_{i}v_{i}$ and
$s_{i,j}=v_{i}u_{j}-u_{i}v_{j}$ for $1\leq i<j\leq m$.
Up to re-indexing the $s_{i,j}$'s, one may write $\frac{q\times q_{a}}{q_{a}^2}=\sum_{j=1}^n\left(\frac{s_j}{q_a}\right)^2$.
Let us note that any $\frac{s_{j}}{d\times q_a}$ is a rational
function which is continuous everywhere. Indeed, since any $u_i$ is
divisible by $d$ in $\R[x,y]$, it suffices to show the continuity at
the point $a$. We use then that $q_a$ is positive definite at $a$ (see
Lemma \ref{lem-polaire}) and just look at the order of each $s_i$ at $a$. 
The polynomial $ s=q\times q_a $ is obviously positive definite since $q$ is.
With $e=d\times q_a$ we get then 
$\frac{s}{e^2}=\sum_{i=1}^n\left(\frac{s_i}{e}\right)^2$ where $\frac{s}{e^2}(a)\not=0$.\par 

Again by Lemma \ref{lem-polaire} and since $q$ is positive definite and $f(a)\not=0$, then $p_a$, the smallest degree homogeneous component at $a$, is proportional to $q_a$ : namely $p_a=f(a)q_a$. Since we assumed that $f(a)=1$, we get $p_a=q_a$. \par

We proceed as previously to write
$$\frac{p\times q_{a}}{q_{a}^2}=\sum_{i=1}^{l}\left(\frac{r_i}{q_a}\right)^2$$
where the $r_i$'s are polynomials such that any $\frac{r_i}{e}$ is a rational function continuous on $\R^2$.

To end, we recall that 
$s=q\times q_a$, $r=p\times p_a$, $e=d\times q_a$ and hence
$$f=\frac{\dfrac{p}{d^2}}{\dfrac{q}{d^2}}=\frac{\dfrac{r}{e^2}}{\dfrac{s}{e^2}}.$$

Since $\frac{s}{e^2}=\frac{1}{q_a}\times\frac{q}{d^2}$,
we obviously have $\Z_f\left(\frac{s}{e^2}\right)\subset \Z_f\left(\frac{q}{d^2}\right)$ and, remembering that  
$\frac{s}{e^2}(a)\not=0$, we get
$a\notin Z_f\left(\frac{s}{e^2}\right)$.
\end{proof}

We are ready then to state our Positivstellensatz without denominator

\begin{thm}
\label{SOS}
Any non-negative function $f\in\SR^0_{[1]}(\R^2)$ can be written as a sum of squares $f=\sum_i^rf_i^2$ where $f_i\in\SR^0_{[1]}(\R^2)$. \par
Moreover, if $f=p/q$ where $p$ and $q$ are sums of squares of
polynomials and $q$ is locally positive definite everywhere, then the denominator of the $f_i$'s can be chosen to be $q\times\prod_{a\in\Zd_f(q)}q_a$ where
$q_a$ is the minimal degree homogeneous component of $q$ at the point $a$.
\end{thm}
Our proof can be seen as an algorithm which gives explicitely a way to
construct our sum of squares (assuming we know how to write a
non-negative polynomials as a sums of squares of regulous functions!).\par
Here is an example where one can see how works the first step :
\begin{ex}
Let $f=p/q$ with
$$p=(x+y)^2+(x-y+y^2)^2$$ and $$q=(x+y^2)^2+y^2.$$
One has $\mathcal Z(q)=\{o\}$ and at $o$, one
has $$p\sim_op_o=(x+y)^2+(x-y)^2=2(x^2+y^2)$$
and $$q\sim_oq_o=x^2+y^2.$$
Consequently, $Z_f(q)=\{o\}.$

Looking at the numerator, one has
$$\begin{array}{rcl}
\dfrac{p\times p_o}{p_o^2}&=&\dfrac{((x+y)^2+(x-y+y^2)^2)((x+y)^2+(x-y)^2)}{p_o^2}\\
&=&\dfrac{(2(x^2+y^2)+y^2(x-y))^2+(y^2(x+y))^2}{p_o^2}\\
&=&\left(1+\dfrac{y^2(x-y)}{2(x^2+y^2)}\right)^2+\left(\dfrac{y^2(x+y)}{2(x^2+y^2)}\right)^2\\
&=&\left(\dfrac{p_1}{p_o}\right)^2+\left(\dfrac{p_2}{p_o}\right)^2=\dfrac{1}{4}\times \left(\left(\dfrac{p_1}{q_o}\right)^2+\left(\dfrac{p_2}{q_o}\right)^2\right)
\end{array}$$
where $p_1=2(x^2+y^2)+y^2(x-y)$ and $p_2=y^2(x+y)$.

Likewise, for the denominator 
$$\begin{array}{rcl}
\dfrac{q\times q_o}{q_o^2}&=&\dfrac{((x+y^2)^2+y^2)(x^2+y^2)}{q_o^2}\\
&=&\left(1+\dfrac{xy^2}{x^2+y^2}\right)^2+\left(\dfrac{y^3}{x^2+y^2}\right)^2\\
&=&\left(\dfrac{q_1}{q_o}\right)^2+\left(\dfrac{q_2}{q_o}\right)^2
\end{array}$$
where $q_1=x^2+y^2+xy^2$ and $q_2=y^3$.\par
Then,
$$\begin{array}{rcl} 
f &= \dfrac{p}{q}=2\times\dfrac{\left(\dfrac{p\times
      p_o}{p_o^2}\right)}{\left(\dfrac{q\times q_o}{q_o^2}\right)}=
2\times \dfrac{1}{4}\times
\dfrac{\left(\dfrac{p_1}{q_o}\right)^2+\left(\dfrac{p_2}{q_o}\right)^2}{\left(
    \dfrac{q}{q_o}\right)^2}\times
\left(\left(\dfrac{q_1}{q_o}\right)^2+\left(\dfrac{q_2}{q_o}\right)^2\right)\\
&=
\dfrac{1}{2}\times \left(\left(\dfrac{p_1q_1}{q_o q}\right)^2+\left(\dfrac{p_1q_2}{q_o q}\right)^2
+\left(\dfrac{p_2q_1}{q_o q}\right)^2+\left(\dfrac{p_2q_2}{q_o
    q}\right)^2\right)
\end{array}$$

A sum of squares identity with denominator $$q_o\times q=(x^2+y^2)\times((x+y^2)^2+y^2).\quad\square$$
\end{ex}

\begin{proof}[Proof of Theorem \ref{SOS}]
We start with $f=p/q\in\SR^0_{[1]}(\R^2) $ where $p$ and $q$ are
polynomials which one may assume non-negative on $\R^2$ and such that
$q$ is definite on $\R^2$. Up to
multiplying both numerator and denominator by a common positive sum of
squares (for instance $(1+x^2+y^2)^N$ for some $N\gg 0$ by Remark \ref{unif_denom}), one may assume that $p$ and $q$ are sums of squares of polynomials. Doing this, $q$ remains locally positive definite at any point of $\R^2$.\par
We argue by induction on the cardinality of $\Z_f(q)$ which is a finite subset of $\R^2$ since $q$ is the denominator of a regulous function on $\R^2$ multiplied by a non-vanishing polynomial ; we set $\Zd_f(q)=\{a_1,\ldots,a_n\}$. 
The case $n=0$ is exactly Proposition \ref{plateapresuneclat}.
\par
The key induction is given by Lemma \ref{keyinductionsos}. Namely, we start with $f=p/q$ together with the sums of squares identities for $p$ and $q$ and we set $d=1$. We take the point $a\in \Zd_f(q)$ and after one application of Lemma \ref{keyinductionsos},
$p$ is replaced by $p\times q_a$, $q$ is replaced by $q\times q_a$ 
and $d$ is replaced by $d\times q_a$. Then we get new sums of squares identities for the new $p$ and $q$ (namely $p\times q_a$ and $q\times q_a$) and overall the new set $\Zd_f$, which no more contains $a$, has cardinality $n-1$.\par
After $n-1$ more applications of Lemma \ref{keyinductionsos}, we get the existence
of some integers $k,m$, some polynomials $r_i$, $s_j$, a sum of squares of polynomials $d$ which is locally positive definite everywhere, and an identity 
$f=\frac{r}{s}$, where $r=\sum_{i=1}^{k} r_i^2$ and $s=\sum_{j=1}^{m} s_j^2$ such that 
\begin{itemize}
\item $s$ is positive definite at any point,
\item any $\frac{r_i}{d}$ and $\frac{s_j}{d}$ is in $\SR^0_{[1]}(\R^2)$,
\item $\Z_f\left(\frac{s}{d^2}\right)= \emptyset$. 
\end{itemize}
It suffices then to write
$$f=\frac{r}{s}=\frac{\frac{r}{d^2}\times\frac{s}{d^2}}{\left(\frac{s}{d^2}\right)^2}=\sum_{i,j}\left(\frac{\frac{r_i}{d}\times \frac{s_j}{d}}{\frac{s}{d^2}}\right)^2=
\sum_{i,j}f_{i,j}^2$$
where $f_{i,j}=\frac{{r_i}\times {s_j}}{s}$.
If $a\in \Z(q)$ and $a\notin\{a_1,\ldots,a_n\}$, then $f(a)=0$ and hence any $f_{i,j}$ is continuous at $a$. Besides, at a point $a_l$, the rational function  
$f_{i,j}=\frac{\frac{r_i}{d}\times \frac{s_j}{d}}{\frac{s}{d^2}}$ is also continuous since 
$\frac{r_i}{d}$, $\frac{s_j}{d}$ and $\frac{s}{d^2}$ are continuous with $\frac{s}{d^2}(a_l)\not=0$.\par
This shows that any $f_{i,j}$ is in $\SR_{[1]}^0(\R^2)$.
\par
Moreover, the common denominator $s$ in the resulting sum of squares can be chosen as
 $$s= q_{a_1}\times q_{a_2}\times\ldots\times q_{a_n}\times q.$$
\end{proof}


\end{document}